\documentclass{article}

\usepackage[margin=1.18in]{geometry}
\usepackage{setspace}

\usepackage[T1]{fontenc}
\usepackage[utf8]{inputenc}
\usepackage{microtype}
\usepackage{xcolor}
\usepackage{float}
\usepackage{placeins}
\usepackage{caption}

\usepackage{booktabs}
\usepackage{
  amsmath,
  amssymb,
  amsthm,
  mathtools,
  mathrsfs
}

\numberwithin{equation}{section}

\newtheorem{theorem}{Theorem}[section]
\newtheorem{proposition}[theorem]{Proposition}
\newtheorem{lemma}[theorem]{Lemma}
\newtheorem{corollary}[theorem]{Corollary}

\theoremstyle{definition}
\newtheorem{definition}[theorem]{Definition}

\theoremstyle{remark}
\newtheorem{remark}[theorem]{Remark}

\usepackage{enumitem}

\usepackage{graphicx}
\usepackage{placeins}
\usepackage{flafter}


\DeclareMathOperator{\diag}{diag}

\newcommand{\Sph}{\mathbb{S}}
\newcommand{\R}{\mathbb{R}}
\newcommand{\Cbb}{\mathbb{C}}

\newcommand{\Z}{\mathbb{Z}}

\newcommand{\dd}{\,\mathrm{d}}

\newcommand{\iso}{\mathrm{iso}}

\newcommand{\Nspace}{\mathcal{N}}
\newcommand{\Bfun}{\mathrm{B}}

\newcommand{\kap}{\kappa}

\newcommand{\SO}{\mathrm{SO}}

\usepackage{xurl}
\usepackage{hyperref}

\hypersetup{
  colorlinks = true,
  linkcolor  = black,
  citecolor  = black,
  urlcolor   = black
}

\DeclareRobustCommand{\doilink}[1]{%
  \href{https://doi.org/#1}{doi:\nolinkurl{#1}}%
}

\title{A Deformation Approach to Axially Symmetric Kernels on the Sphere:
Harmonic Coupling and Regularity}

\author{%
Janin Jäger,
Fabian Stallbauer\thanks{\texttt{fabian.stallbauer@ku.de}}\\[0.5em]
\small Catholic University of Eichstätt--Ingolstadt, Germany
}
\date{}
\begin{document}

\maketitle

%% ============================================================
%% Abstract
%% ============================================================

\begin{abstract}

In this paper, we construct axially symmetric kernels on the sphere.
We start by taking an isotropic kernel and transform it into an axially symmetric kernel via a specific deformation.
For a general deformation, an explicit spherical-harmonic
representation is not necessarily available. For our deformation, we can obtain
such a representation of the deformed kernel in terms of the expansion coefficients
of the original isotropic kernel. We derive this representation by explicitly describing the effect of the deformation on the spherical harmonic basis. From this, an explicit finite
sum formula for the coupling coefficients is deduced and we show that, for each fixed
input mode, the coefficients decay geometrically at a rate controlled by
our deformation parameter. We further note that if the native space of the
isotropic kernel is norm equivalent to a Sobolev-space, then the deformed kernel has the same native space with an equivalent norm. In our numerical experiments,
we show that our deformation can improve prediction when the underlying
random field has an axially symmetric covariance.

\end{abstract}


\medskip

%% ============================================================
%% Section: Introduction
%% ============================================================
\section{Introduction}
\label{sec:introduction}

{
Positive definite kernels play a central role in harmonic analysis and
spherical approximation theory \cite{AtkinsonHan2012,DaiXu2013}, as well as
in statistics, where they serve as covariance functions, e.g. for global
spatial data \cite{JeongJunGenton2017}. In the latter setting, the curvature
of the domain under consideration is essential, since replacing Euclidean
distance by great-circle distance preserves positive definiteness for
some models but not for others. Examples are discussed in
\cite{HuangZhangRobeson2011}. A standard setting is to assume isotropy of the kernel.
Such kernels depend only on the spherical distance between their arguments.

For continuous isotropic kernels on \(\Sph^2\), positive definiteness
is equivalent to having nonnegative Legendre coefficients
\cite{Schoenberg1942}. For strict positive definiteness, additional
conditions on the positive coefficients are needed
\cite{ChenMenegattoSun2003,Gneiting2013}.
In this isotropic setting, the resulting coefficient structure in a
spherical-harmonic basis is diagonal. Furthermore, the decay of those
spectral coefficients also controls the regularity of the associated
native space, which for important classes of spherical kernels is
equivalent to a spherical Sobolev space \cite{HubbertLeGiaMorton2015}.
}

{
Nevertheless, full isotropy can be too restrictive when the data have
a preferred direction. An intermediate model class consists of axially symmetric
kernels, which are invariant under rotations about one axis. Their harmonic
description dates back to \cite{Jones1963}, where the diagonal isotropic
structure is replaced by a coefficient block structure. Related work
gives similar representations and formulations
\cite{HuangZhangRobeson2012,BuhmannJaeger2022IMA} and studies conditions
for strict positive definiteness and other properties of axially
symmetric kernels \cite{BissiriPeronPorcu2020}.

For particular model classes, explicit spherical-harmonic representations
are also available. Hitczenko and Stein
\cite{HitczenkoStein2012} consider Gaussian processes whose covariance is
invariant under shifts in longitude and which are obtained by applying
first-order differential operators to an isotropic process.
}

{
Another way to construct such kernels is to use a suitable deformation
of the sphere and evaluate an isotropic kernel at the transformed points
\cite{PorcuSenoussiMendozaBevilacqua2020}. This, however, does not generally
give us explicit formulas for the corresponding spherical-harmonic
coefficients, which are useful tools for truncated approximations and
field simulations \cite{PorcuCastruccioAlegriaCrippa2019}.
Our construction gives both: an explicit spherical-harmonic series for the deformed kernel and a closed-form expression whenever one is available for the original kernel.
}

{
Sampson and Guttorp \cite{SampsonGuttorp1992} used domain deformations to model nonstationary dependence. A similar principle appears in approximation theory, where transformations of points or bases underlie anisotropically transformed radial basis functions \cite{BeatsonDavydovLevesley2010}, as well as mapped bases and fake-node constructions \cite{DeMarchiMarchettiPerracchionePoggiali2020MappedBases,DeMarchiMarchettiPerracchionePoggiali2021FakeNodes}.
}
{
Here, we study a normalized axial deformation obtained by scaling points
along a chosen axis in \(\R^3\) and normalizing them back to \(\Sph^2\).
With this deformation, we can evaluate the kernel directly, derive its
spherical-harmonic representation, and study properties of the kernel
and its native space. The paper is organized as follows.

Section 2 introduces the definitions and background for isotropic and
axially symmetric kernels. In Section 3, we start with the basic properties of pullback kernels, using standard kernel theory to study positive definiteness and native spaces. Section 4 then introduces our deformation. We derive the harmonic coupling coefficients, estimate how quickly they decay, and extend the construction to an arbitrary axis. We also show that the deformation preserves the Sobolev regularity of the native space and explain how interpolation error bounds carry over from the original kernel. Section 5 gives examples of the deformed kernels and compares their predictive performance with that of isotropic kernels, both when the covariance model is correctly specified and when it is misspecified.
}

%% ============================================================
%% Section: Preliminaries
%% ============================================================
\section{Preliminaries}
\label{sec:framework}

This section introduces the notation required for the harmonic representations studied in this paper including spherical harmonics, Sobolev
spaces, and positive definite kernels with their associated native spaces.

%% ------------------------------------------------------------
%% Subsection: Spherical harmonics and Sobolev spaces
%% ------------------------------------------------------------
\subsection{Spherical harmonics and Sobolev spaces}
\label{subsec:harmonics-sobolev}

We use the definitions of the associated Legendre
functions and complex spherical harmonics as in
\cite[Eqs.~14.7.8 and~14.30.1]{NIST:DLMF}.
For integers \(\ell\ge0\) and \(0\le m\le\ell\), define
\[
 P_\ell^m(x)
 =
 (-1)^m(1-x^2)^{m/2}
 \frac{d^m}{dx^m}P_\ell(x),
 \qquad -1\leq x\leq 1,
\]
where \(P_\ell\) is the Legendre polynomial of degree \(\ell\).

{
We use the standard Cartesian realization
\[
 \Sph^2
 =
 \{s=(s_1,s_2,s_3)\in\R^3:\|s\|_2=1\}.
\]
For harmonic analysis, we use the spherical parametrization
\[
 s=\omega(L,\lambda)
 :=
 \bigl(
 \sin L\cos\lambda,\,
 \sin L\sin\lambda,\,
 \cos L
 \bigr),
\]
where \(L\in[0,\pi]\) is the polar angle and
\(\lambda\in[0,2\pi)\) is the longitude. For
\(s=\omega(L,\lambda)\) and \(0\le m\le\ell\), set
\begin{equation}
\label{eq:Ylm-def-positive}
 Y_{\ell m}(s)
 =
 \mathcal N_{\ell m}
 P_\ell^m(\cos L)e^{im\lambda},
\end{equation}
}
where
\begin{equation}
\label{eq:Ylm-normalization}
 \mathcal N_{\ell m}
 =
 \left(
 \frac{2\ell+1}{4\pi}
 \frac{(\ell-m)!}{(\ell+m)!}
 \right)^{1/2}.
\end{equation}
For negative orders, we use
\begin{equation}
\label{eq:Ylm-negative}
 Y_{\ell,-m}
 =
 (-1)^m\overline{Y_{\ell m}},
 \qquad 1\le m\le\ell.
\end{equation}

Let \(d\omega\) be the standard surface-area measure on \(\Sph^2\), so that
\[
\int_{\Sph^2} 1\,\dd\omega = 4\pi.
\]
The inner product on \(L^2(\Sph^2)\) is
\[
 \langle f,g\rangle_{L^2(\Sph^2)}
 =
 \int_{\Sph^2}
 f(s)\overline{g(s)}\,\dd\omega(s).
\]
With this normalization,
\[
 \{Y_{\ell m}:\ell\ge0,\ -\ell\le m\le\ell\}
\]
is an orthonormal basis of \(L^2(\Sph^2)\). Each \(Y_{\ell m}\) is an
eigenfunction of the Laplace--Beltrami operator:
\[
 -\Delta_{\Sph^2}Y_{\ell m}
 =
 \Lambda_\ell Y_{\ell m},
 \qquad
 \Lambda_\ell=\ell(\ell+1).
\]

For \(f\in L^2(\Sph^2)\), define the Fourier--Laplace coefficients by
\[
 \widehat f_{\ell m}
 =
 \langle f,Y_{\ell m}\rangle_{L^2(\Sph^2)}.
\]
For \(\tau\ge0\), the Sobolev space \(H^\tau(\Sph^2)\) consists of all
functions for which
\begin{equation}
\label{eq:sobolev-S2}
 \|f\|_{H^\tau(\Sph^2)}^2
 =
 \sum_{\ell=0}^{\infty}
 \sum_{m=-\ell}^{\ell}
 (1+\Lambda_\ell)^\tau
 |\widehat f_{\ell m}|^2
 <\infty.
\end{equation}
This is the standard spectral definition of Sobolev spaces on the sphere;
see \cite{HubbertLeGiaMorton2015}. For \(\tau>0\), the spaces
\(H^{-\tau}(\Sph^2)\) are defined by duality; see \cite{Taylor2011}.

%% ------------------------------------------------------------
%% Subsection: Isotropic kernels and native spaces
%% ------------------------------------------------------------
\subsection{Isotropic kernels and native spaces}
\label{subsec:isotropic-kernels}

Unless stated otherwise, all kernels are complex-valued and Hermitian:
\[
K(s,t)=\overline{K(t,s)}.
\]
A Hermitian kernel \(K:X\times X\to\Cbb\) is positive definite if, for every
\(N\in\mathbb N\), every \(s_1,\ldots,s_N\in X\), and every
\(c_1,\ldots,c_N\in\Cbb\),
\[
\sum_{i,j=1}^{N}
c_i\overline{c_j}\,K(s_i,s_j)\ge0.
\]
It is strictly positive definite if the inequality is strict whenever the
points \(s_1,\ldots,s_N\) are pairwise distinct and
\((c_1,\ldots,c_N)\neq0\).

Every positive definite kernel \(K\) induces a unique
reproducing kernel Hilbert space of functions on \(X\). We denote this space
by \(\Nspace(K)\) and refer to it as the native space of \(K\); see
\cite{WendlandSDA}.

For normed spaces \(X\) and \(Y\), we write \(X\equiv Y\) when the spaces coincide and their
norms are equivalent.

A continuous kernel
\(K_{\iso}:\Sph^2\times\Sph^2\to\Cbb\) is isotropic if it depends only on
the Euclidean inner product of its arguments:
\[
K_{\iso}(s,t)=\psi(s\cdot t)
\]
for a continuous function \(\psi:[-1,1]\to\R\). 
By Schoenberg's theorem \cite{Schoenberg1942},
every continuous isotropic positive definite kernel on \(\Sph^2\)
has a Legendre expansion with nonnegative coefficients;
see also \cite{Gneiting2013,BergPorcu2017}.

With the spherical-harmonic normalization from
Subsection~\ref{subsec:harmonics-sobolev}, this expansion takes the form
\begin{equation}
\label{eq:Kiso-expansion-S2}
K_{\iso}(s,t)
=
\sum_{\ell=0}^{\infty}
\kap_\ell
\sum_{m=-\ell}^{\ell}
Y_{\ell m}(s)\overline{Y_{\ell m}(t)},
\end{equation}
where \(\kap_\ell\ge0\) denotes the spectral coefficient. Using the spherical-harmonic addition
theorem \cite[Eq.~14.30.9]{NIST:DLMF}
\[
\sum_{m=-\ell}^{\ell}
Y_{\ell m}(s)\overline{Y_{\ell m}(t)}
=
\frac{2\ell+1}{4\pi}P_\ell(s\cdot t),
\]
we obtain the equivalent Legendre representation
\[
K_{\iso}(s,t)
=
\sum_{\ell=0}^{\infty}
\kap_\ell\frac{2\ell+1}{4\pi}P_\ell(s\cdot t),
\]
with the summability condition
\begin{equation}
\label{eq:schoenberg-summability-kappa}
\sum_{\ell=0}^{\infty}(2\ell+1)\kap_\ell<\infty.
\end{equation}

This condition yields absolute and uniform convergence of the Legendre
expansion.
We use the following standard spectral criterion; see
\cite{SloanSommariva2008,WendlandSDA}.

\begin{proposition}
\label{prop:isotropic-native-space-spectral}
Let \(K_{\iso}\) be given by
\eqref{eq:Kiso-expansion-S2}, and assume that
\(\kap_\ell>0\) for every \(\ell\ge0\). Suppose that there are constants
\(0<c_-\le c_+<\infty\) and \(\tau>1\) such that
\begin{equation}
\label{eq:two-sided-spectral-assumption}
c_-(1+\Lambda_\ell)^{-\tau}
\le
\kap_\ell
\le
c_+(1+\Lambda_\ell)^{-\tau},
\qquad \ell\ge0.
\end{equation}
Then
{
\[
\Nspace(K_{\iso})\equiv H^\tau(\Sph^2).
\]
}
More precisely,
\[
\frac{1}{c_+}
\|g\|_{H^\tau(\Sph^2)}^2
\le
\|g\|_{\Nspace(K_{\iso})}^2
\le
\frac{1}{c_-}
\|g\|_{H^\tau(\Sph^2)}^2.
\]
\end{proposition}

%% ------------------------------------------------------------
%% Subsection: Axial symmetry and Jones blocks
%% ------------------------------------------------------------
\subsection{Axial symmetry and Jones blocks}
\label{subsec:jones-blocks}

Axial symmetry requires invariance only under rotations about a chosen
axis. Let
\[
G_{e_3}
:=
\{R\in\SO(3):Re_3=e_3\}
\]
denote the subgroup of rotations that fixes the reference axis \(e_3\). A
kernel \(K:\Sph^2\times\Sph^2\to\Cbb\) is axially symmetric about \(e_3\) if
\[
K(Rs,Rt)=K(s,t)
\]
for all \(s,t\in\Sph^2\) and \(R\in G_{e_3}\).

Throughout this subsection, we assume that \(K\) is continuous.
We write
\[
s=\omega(L,\lambda),
\qquad
t=\omega(L',\lambda').
\]
If we rotate both points by an angle \(\alpha\) about \(e_3\), we add \(\alpha\) to both longitudes, so their difference \(\lambda-\lambda'\) stays the same. An axially symmetric kernel is unchanged by this rotation and depends on the longitudes only through their difference. We can therefore write
\[
K(s,t)=\Psi(L,L',\lambda-\lambda'),
\]
as in the classical coordinate form used in
\cite{Jones1963,HuangZhangRobeson2012}.

The corresponding spectral representation for axially symmetric processes on
the sphere was introduced in \cite{Jones1963}. We use the coefficient notation
of \cite{BissiriPeronPorcu2020}; a related reformulation is given in
\cite{HuangZhangRobeson2012}. In this notation, the representation takes the
form
\begin{equation}
\label{eq:jones-representation}
K(s,t)
=
\sum_{m\in\mathbb Z}
\sum_{n,n'\ge |m|}
e^{im(\lambda-\lambda')}
P_n^{|m|}(\cos L)
P_{n'}^{|m|}(\cos L')
c_m(n,n').
\end{equation}

To express the same structure in a spherical-harmonic basis, first define
the general coefficients
\begin{equation}
\label{eq:general-harmonic-coefficients}
a_{nn'}^{(m,m')}
:=
\int_{\Sph^2}\int_{\Sph^2}
K(s,t)
\overline{Y_{nm}(s)}
Y_{n'm'}(t)
\,\dd\omega(t)\,\dd\omega(s).
\end{equation}
In the spherical-harmonic basis, axial symmetry is characterized by 
\begin{equation}
\label{eq:axial-selection-rule}
a_{nn'}^{(m,m')}=0
\qquad
\text{whenever }m\neq m',
\end{equation}
as shown in \cite{BuhmannJaeger2022}. We therefore abbreviate the
remaining coefficients by
\begin{equation}
\label{eq:fixed-order-coefficients}
a_{nn'}^{(m)}
:=
a_{nn'}^{(m,m)}.
\end{equation}

The spherical-harmonic expansion of an axially symmetric kernel consequently
takes the form
\begin{equation}
\label{eq:axial-harmonic-expansion}
K(s,t)
=
\sum_{m\in\mathbb Z}
\sum_{n,n'\ge |m|}
a_{nn'}^{(m)}
Y_{nm}(s)
\overline{Y_{n'm}(t)},
\end{equation}
with convergence in \(L^2(\Sph^2\times\Sph^2)\).

The expansions
\eqref{eq:jones-representation} and
\eqref{eq:axial-harmonic-expansion}
describe the same harmonic structure with different normalizations. Setting
\(\mu=|m|\) and using the spherical-harmonic convention fixed above gives
\begin{equation}
\label{eq:jones-harmonic-relation}
c_m(n,n')
=
\mathcal N_{n\mu}\mathcal N_{n'\mu}
a_{nn'}^{(m)}.
\end{equation}

For convenience, throughout this paper we refer to the
coefficient matrix
\begin{equation}
\label{eq:jones-block}
A^{(m)}
:=
\bigl(a_{nn'}^{(m)}\bigr)_{n,n'\ge |m|}
\end{equation}
as the \(m\)-th Jones block. 
%% ------------------------------------------------------------
%% Subsection: Gegenbauer polynomials and associated Legendre functions
%% ------------------------------------------------------------
\subsection{Gegenbauer polynomials and associated Legendre functions}
\label{subsec:gegenbauer}

For \(q\in\mathbb N_0\) and \(\nu>0\), the Gegenbauer polynomial
\(C_q^{(\nu)}\) admits the finite expansion
\begin{equation}
\label{eq:Gegenbauer-finite-expansion}
C_q^{(\nu)}(x)
=
\sum_{k=0}^{\lfloor q/2\rfloor}
(-1)^k
\frac{(\nu)_{q-k}}
{k!(q-2k)!}
(2x)^{q-2k},
\end{equation}
as recorded in
\cite[Eq.~18.5.10]{NIST:DLMF}.
Here
\[
 (a)_r=\frac{\Gamma(a+r)}{\Gamma(a)},
 \qquad r\in\mathbb N_0,
\]
denotes the Pochhammer symbol. With the normalization recorded in
\cite[Table~18.6.1]{NIST:DLMF},
\begin{equation}
\label{eq:Gegenbauer-normalization}
C_q^{(\nu)}(1)
=
\frac{\Gamma(q+2\nu)}
{\Gamma(q+1)\Gamma(2\nu)}.
\end{equation}
For integers \(\mu\ge0\) and \(n\ge\mu\), the associated Legendre
functions satisfy
\begin{equation}
\label{eq:associated-Legendre-Gegenbauer}
 P_n^\mu(x)
 =
 d_\mu(1-x^2)^{\mu/2}
 C_{n-\mu}^{(\mu+\frac12)}(x),
 \qquad
 d_\mu
 =
 (-1)^\mu\frac{(2\mu)!}{2^\mu\mu!},
 \qquad -1\le x\le1.
\end{equation}
This identity is given in
\cite[Eq.~18.11.1]{NIST:DLMF}.

For integers \(a\ge\mu\ge0\), define
\begin{subequations}
\label{eq:associated-Legendre-finite-data}
\begin{align}
 R_{a,\mu}
 &=
 \left\lfloor\frac{a-\mu}{2}\right\rfloor,
 \qquad
 \gamma_{a,\mu,k}
 =
 \frac{(-1)^{\mu+k}(2a-2k)!}
 {2^a k!(a-k)!(a-\mu-2k)!},
 \quad
 0\le k\le R_{a,\mu},
 \label{eq:gamma-def-main}
 \\
 P_a^\mu(x)
 &=
 (1-x^2)^{\mu/2}
 \sum_{k=0}^{R_{a,\mu}}
 \gamma_{a,\mu,k}x^{a-\mu-2k},
 \qquad -1\le x\le1.
 \label{eq:associated-Legendre-finite}
\end{align}
\end{subequations}
The finite representation
\eqref{eq:associated-Legendre-finite} follows by substituting
\(q=a-\mu\) and \(\nu=\mu+\frac12\) into
\eqref{eq:Gegenbauer-finite-expansion} and then using
\eqref{eq:associated-Legendre-Gegenbauer}. Appendix~\ref{app:explicit-S2}
gives the short calculation.

%% ============================================================
%% Section: Kernel pullbacks
%% ============================================================
\section{Kernel pullbacks}
\label{sec:general-pullbacks}

{
In this section, we present a general framework for kernel pullbacks,
covering positive definiteness, symmetry, native spaces, and Sobolev
regularity. Although most of these results are consequences of standard
kernel theory, we include short proofs for completeness, as we will
use these results in the later analysis.
}

%% ------------------------------------------------------------
%% Subsection: Positive definiteness under pullback
%% ------------------------------------------------------------
\subsection{Positive definiteness under pullback}
\label{subsec:positive-definite-pullbacks}

Let \(X\) and \(S\) be nonempty sets, let
\(K:X\times X\to\Cbb\) be a kernel, and let \(\Phi:S\to X\) be any map.
The pullback of \(K\) by \(\Phi\) is
\begin{equation}
\label{eq:abstract-pullback-kernel}
 K_\Phi(s,t)
 :=
 K\bigl(\Phi(s),\Phi(t)\bigr),
 \qquad s,t\in S.
\end{equation}

\begin{proposition}
\label{prop:pd-pullback}
Let \(K:X\times X\to\Cbb\) be positive definite and let
\(\Phi:S\to X\) be any map. Then \(K_\Phi\) is positive definite on \(S\).
If \(K\) is strictly positive definite and \(\Phi\) is injective, then
\(K_\Phi\) is strictly positive definite.
\end{proposition}

\begin{proof}
For \(s_1,\ldots,s_N\in S\) and \(c_1,\ldots,c_N\in\Cbb\),
\[
 \sum_{i,j=1}^{N}c_i\overline{c_j}\,K_\Phi(s_i,s_j)
 =
 \sum_{i,j=1}^{N}c_i\overline{c_j}\,
 K\bigl(\Phi(s_i),\Phi(s_j)\bigr)
 \ge0.
\]
If \(s_1,\ldots,s_N\) are pairwise distinct, their
images under the injective map \(\Phi\) are also pairwise distinct.
Since \(K\) is strictly positive definite, the sum is strictly
positive whenever the coefficient vector is nonzero.
\end{proof}

%% ------------------------------------------------------------
%% Subsection: Native spaces under pullback
%% ------------------------------------------------------------
\subsection{Native spaces under pullback}
\label{subsec:general-native-space-pullback}

For an arbitrary map \(\Phi:S\to X\), the standard RKHS pullback theorem
gives
\begin{equation}
\label{eq:general-rkhs-pullback-set}
 \Nspace(K_\Phi)
 =
 \bigl\{g\circ\Phi:g\in\Nspace(K)\bigr\},
\end{equation}
where the norm is the minimum \(\Nspace(K)\)-norm over all representations
of a function as \(g\circ\Phi\); see
\cite[Ch.~5]{PaulsenRaghupathi2016}. If \(\Phi\) is bijective, this
representation is unique.

\begin{corollary}
\label{cor:bijective-rkhs-pullback}
Let \(K:X\times X\to\Cbb\) be positive definite and let
\(\Phi:S\to X\) be bijective. Then
\[
 C_\Phi:\Nspace(K)\to\Nspace(K_\Phi),
 \qquad
 C_\Phi g=g\circ\Phi,
\]
is an isometric isomorphism. In particular,
\begin{equation}
\label{eq:bijective-rkhs-pullback-isometry}
 \|g\circ\Phi\|_{\Nspace(K_\Phi)}
 =
 \|g\|_{\Nspace(K)},
\end{equation}
and
\[
 C_\Phi^{-1}=C_{\Phi^{-1}}.
\]
\end{corollary}

\begin{proof}
The general pullback theorem gives the set identity. Since \(\Phi\) is
bijective, \(h=g\circ\Phi\) determines
\(g=h\circ\Phi^{-1}\) uniquely. The quotient norm therefore reduces to
\eqref{eq:bijective-rkhs-pullback-isometry}. The set identity also gives
surjectivity, and composition with \(\Phi^{-1}\) gives the inverse.
\end{proof}

%% ------------------------------------------------------------
%% Subsection: Symmetry under pullback
%% ------------------------------------------------------------
\subsection{Symmetry under pullback}
\label{subsec:symmetry-equivariance}

We use the definitions of group actions and equivariant
maps from \cite{Lee2003}.

\begin{definition}[\(G\)-invariant kernels and \(G\)-equivariant maps]
\label{def:invariance-equivariance}
Let a group \(G\) act on sets \(S\) and \(X\). A kernel
\(K:X\times X\to\Cbb\) is called \(G\)-invariant if
\[
 K(gx,gy)=K(x,y)
 \qquad
 \text{for all }g\in G\text{ and }x,y\in X.
\]
A map \(\Phi:S\to X\) is called \(G\)-equivariant if
\[
 \Phi(gs)=g\Phi(s)
 \qquad
 \text{for all }g\in G\text{ and }s\in S.
\]
\end{definition}

For the axial symmetry introduced in
Subsection~\ref{subsec:jones-blocks}, the relevant group is
\(G_{e_3}\), the rotations that fix the reference axis \(e_3\). Thus
\(G_{e_3}\)-invariance is exactly axial symmetry about \(e_3\).

\begin{proposition}
\label{prop:symmetry-equivariant-pullback}
Let \(G\) act on \(S\) and \(X\). If \(K:X\times X\to\Cbb\) is
\(G\)-invariant and \(\Phi:S\to X\) is \(G\)-equivariant, then the
pullback kernel \(K_\Phi\) is \(G\)-invariant on \(S\).
\end{proposition}

\begin{proof}
For \(g\in G\) and \(s,t\in S\),
\[
 \begin{aligned}
 K_\Phi(gs,gt)
 &=
 K\bigl(\Phi(gs),\Phi(gt)\bigr)\\
 &=
 K\bigl(g\Phi(s),g\Phi(t)\bigr)\\
 &=
 K\bigl(\Phi(s),\Phi(t)\bigr)\\
 &=
 K_\Phi(s,t).
 \end{aligned}
\]
\end{proof}

%% ------------------------------------------------------------
%% Subsection: Sobolev stability
%% ------------------------------------------------------------
\subsection{Sobolev stability}
\label{subsec:sobolev-stability-general}

{
We now consider \(\Sph^2\). On a compact manifold, composition with a
smooth diffeomorphism preserves Sobolev regularity
\cite[Ch.~4]{Taylor2011}. The next proposition states this result for the Sobolev spaces defined in
Subsection~\ref{subsec:harmonics-sobolev}. We give a short proof using
charts in Appendix~\ref{app:native-transport-proof}.
}

\begin{proposition}
\label{prop:sobolev-stability-diffeomorphism}
Let \(\Phi:\Sph^2\to\Sph^2\) be a \(C^\infty\)-diffeomorphism. Then, for
every \(\tau\ge0\),
\[
 C_\Phi:H^\tau(\Sph^2)\to H^\tau(\Sph^2),
 \qquad
 C_\Phi f=f\circ\Phi,
\]
is a bounded isomorphism with inverse \(C_{\Phi^{-1}}\). Equivalently,
\begin{equation}
\label{eq:sobolev-composition-equivalence}
 \|f\circ\Phi\|_{H^\tau(\Sph^2)}
 \asymp_{\Phi,\tau}
 \|f\|_{H^\tau(\Sph^2)}.
\end{equation}
\end{proposition}

\begin{proof}
The integer-order case is proved in
Theorem~\ref{thm:sobolev-stability-appendix}, and interpolation gives the
fractional-order case in
Corollary~\ref{cor:fractional-sobolev-stability-appendix}. Applying the same
argument to \(\Phi^{-1}\) gives the bounded inverse and the two-sided norm
estimate.
\end{proof}

Here \(\asymp_{\Phi,\tau}\) denotes two-sided bounds
with positive constants depending only on \(\Phi\) and \(\tau\),
but not on \(f\).

For later reference, we specialize the pullback in
\eqref{eq:abstract-pullback-kernel} to a positive definite kernel
\(K:\Sph^2\times\Sph^2\to\Cbb\) and write
\begin{equation}
\label{eq:general-pullback-kernel}
 K_\Phi(s,t)
 =
 K\bigl(\Phi(s),\Phi(t)\bigr).
\end{equation}
Combining the RKHS isometry with Sobolev stability gives the following
consequence.

\begin{corollary}
\label{cor:native-transport}
Let \(\Phi:\Sph^2\to\Sph^2\) be a \(C^\infty\)-diffeomorphism, and let
\(K_\Phi\) be defined by \eqref{eq:general-pullback-kernel}.
{
If, for some \(\tau\ge0\),
\[
 \Nspace(K)\equiv H^\tau(\Sph^2),
\]
then
\begin{equation}
\label{eq:native-space-transport-sobolev}
 \Nspace(K_\Phi)\equiv H^\tau(\Sph^2).
\end{equation}
}
\end{corollary}

\begin{proof}
{
By Corollary~\ref{cor:bijective-rkhs-pullback},
\[
\Nspace(K_\Phi)=C_\Phi\Nspace(K),
\]
and \(C_\Phi\) is an isometry between the two native spaces.
Proposition~\ref{prop:sobolev-stability-diffeomorphism} gives
\[
C_\Phi H^\tau(\Sph^2)=H^\tau(\Sph^2)
\]
with equivalent Sobolev norms. These two results imply
\eqref{eq:native-space-transport-sobolev}.
}
\end{proof}

%% ============================================================
%% Section: Normalized axial deformation
%% ============================================================
\section{Normalized axial deformation}
\label{sec:canonical-axial}

We now introduce \(T_\beta\), our $G_{e_3}$-equivariant transformation of the sphere.
We then study its effect on the spherical-harmonic representation
of an isotropic kernel.

Fix \(\beta>0\), and define
\begin{equation}
\label{eq:Abeta-Tbeta-def}
A_\beta=\diag(1,1,\beta),
\qquad
T_\beta(s)=\frac{A_\beta s}{\|A_\beta s\|},
\qquad s\in\Sph^2.
\end{equation}
The map \(T_\beta\) stretches or compresses the third
coordinate by a factor of \(\beta\) and then normalizes the vector
back onto \(\Sph^2\).

Let \(K_{\iso}\) be a continuous isotropic positive definite kernel, and
define the pullback kernel
\begin{equation}
\label{eq:Kbeta-def}
 K_\beta(s,t)
 :=
 K_{\iso}\bigl(T_\beta(s),T_\beta(t)\bigr).
\end{equation}

{
Specializing \eqref{eq:general-harmonic-coefficients} to \(K_\beta\), we define
\begin{equation}
\label{eq:canonical-pullback-coefficients}
a_{nn'}^{(m,m')}(\beta)
:=
\int_{\Sph^2}\int_{\Sph^2}
K_\beta(s,t)
\overline{Y_{nm}(s)}
Y_{n'm'}(t)
\,\dd\omega(t)\,\dd\omega(s),
\end{equation}
for \(n,n'\ge0\), \(|m|\le n\), and \(|m'|\le n'\).
With these coefficients, the expansion of \(K_\beta\) in
\(L^2(\Sph^2\times\Sph^2)\) is
\begin{equation}
\label{eq:Kbeta-general-harmonic-expansion}
K_\beta(s,t)
=
\sum_{n=0}^{\infty}\sum_{m=-n}^{n}
\sum_{n'=0}^{\infty}\sum_{m'=-n'}^{n'}
a_{nn'}^{(m,m')}(\beta)
Y_{nm}(s)\overline{Y_{n'm'}(t)}.
\end{equation}

On the other hand, inserting the spherical-harmonic expansion
\eqref{eq:Kiso-expansion-S2} of \(K_{\iso}\) into
\eqref{eq:Kbeta-def} gives
\begin{equation}
\label{eq:Kbeta-composed-series}
 K_\beta(s,t)
 =
 \sum_{\ell=0}^{\infty}
 \kap_\ell
 \sum_{m=-\ell}^{\ell}
 Y_{\ell m}\bigl(T_\beta(s)\bigr)
 \overline{Y_{\ell m}\bigl(T_\beta(t)\bigr)}.
\end{equation}

In this series, we keep the coefficients \(\kap_\ell\) and
evaluate the harmonics at the transformed points. To find the
coefficients in \eqref{eq:Kbeta-general-harmonic-expansion}, we need
to express these transformed harmonics in the original spherical-harmonic
basis. This will show how the deformation couples different harmonic
degrees. We begin by looking at the geometry of \(T_\beta\).
}
%% ------------------------------------------------------------
%% Subsection: Geometry of the deformation
%% ------------------------------------------------------------
\subsection{Geometry of the deformation}
\label{subsec:canonical-deformation}

We first show that \(T_\beta\) is a smooth diffeomorphism
that commutes with rotations about \(e_3\), so the general pullback results
apply.

\begin{proposition}
\label{prop:diffeo-equivariance}
The map \(T_\beta:\Sph^2\to\Sph^2\) is a smooth diffeomorphism with inverse
\begin{equation}
\label{eq:Tbeta-inverse}
 T_\beta^{-1}(s)
 =
 \frac{A_\beta^{-1}s}{\|A_\beta^{-1}s\|},
 \qquad
 A_\beta^{-1}=\diag(1,1,\beta^{-1}).
\end{equation}
In particular, \(T_\beta^{-1}=T_{\beta^{-1}}\). Moreover, if
\(R\in\SO(3)\) satisfies \(Re_3=e_3\), then
\begin{equation}
\label{eq:Tbeta-equivariance}
 T_\beta(Rs)=R\,T_\beta(s).
\end{equation}
\end{proposition}

\begin{proof}
{
Since \(\beta>0\), the diagonal matrix \(A_\beta\) is invertible with the
inverse displayed in \eqref{eq:Tbeta-inverse}, and \(A_\beta s\neq0\) for
every \(s\in\Sph^2\). Direct substitution gives
\[
 T_{\beta^{-1}}\bigl(T_\beta(s)\bigr)=s,
 \qquad
 T_\beta\bigl(T_{\beta^{-1}}(s)\bigr)=s.
\]
Both maps are smooth, so \(T_\beta\) is a smooth diffeomorphism.
}

If \(Re_3=e_3\), then \(R\) commutes with the projection
\(e_3e_3^\top\), and therefore with
\[
 A_\beta=I_3+(\beta-1)e_3e_3^\top.
\]
Since \(R\) is orthogonal,
\[
 T_\beta(Rs)
 =
 \frac{RA_\beta s}{\|RA_\beta s\|}
 =
 R\,T_\beta(s).
\]
\end{proof}

\begin{corollary}
\label{cor:pullback-axial}
If \(K_{\iso}\) is positive definite,
then \(K_\beta\) is positive definite and axially symmetric about
\(e_3\). If \(K_{\iso}\) is strictly positive definite, then
\(K_\beta\) is strictly positive definite.
\end{corollary}

\begin{proof}
Positive definiteness follows from
Proposition~\ref{prop:pd-pullback}. If \(K_{\iso}\) is strictly positive
definite, strict positive definiteness is preserved because \(T_\beta\) is
injective. Finally, \(K_{\iso}\) is invariant under rotations, and
\(T_\beta\) is equivariant under rotations fixing \(e_3\) by
Proposition~\ref{prop:diffeo-equivariance}. Axial symmetry therefore follows
from Proposition~\ref{prop:symmetry-equivariant-pullback}.
\end{proof}

We next express \(T_\beta\) in spherical coordinates to
make its action on the polar angle and longitude explicit.

Write
\[
 \omega(L,\lambda)
 =
 \begin{pmatrix}
 \sin L\cos\lambda\\
 \sin L\sin\lambda\\
 \cos L
 \end{pmatrix},
 \qquad
 0\le L\le\pi,
 \quad
 0\le\lambda<2\pi.
\]
Then
\[
 A_\beta\omega(L,\lambda)
 =
 \begin{pmatrix}
 \sin L\cos\lambda\\
 \sin L\sin\lambda\\
 \beta\cos L
 \end{pmatrix},
 \qquad
 \|A_\beta\omega(L,\lambda)\|
 =
 \sqrt{\sin^2L+\beta^2\cos^2L}.
\]
Therefore,
\[
 T_\beta\bigl(\omega(L,\lambda)\bigr)
 =
 \frac{1}{\sqrt{\sin^2L+\beta^2\cos^2L}}
 \begin{pmatrix}
 \sin L\cos\lambda\\
 \sin L\sin\lambda\\
 \beta\cos L
 \end{pmatrix}.
\]

Set \(x=\cos L\), and define
\begin{equation}
\label{eq:xtilde-def}
 \widetilde x_\beta(x)
 =
 \frac{\beta x}{\sqrt{1+(\beta^2-1)x^2}},
 \qquad -1\le x\le1.
\end{equation}
The transformed polar coordinate satisfies
\[
 \cos\Theta_\beta(L)
 =
 \frac{\beta\cos L}{\sqrt{\sin^2L+\beta^2\cos^2L}}
 =
 \widetilde x_\beta(\cos L).
\]
Hence, with
\begin{equation}
\label{eq:Theta-def}
 \Theta_\beta(L)
 =
 \arccos\!\bigl(\widetilde x_\beta(\cos L)\bigr),
\end{equation}
we obtain
\begin{equation}
\label{eq:Tbeta-canonical-coordinate}
 T_\beta\bigl(\omega(L,\lambda)\bigr)
 =
 \omega\bigl(\Theta_\beta(L),\lambda\bigr).
\end{equation}
For \(\beta=1\), the map \(T_\beta\) is the identity. If \(\beta>1\),
points away from the poles and the equator move toward the pole in
their hemisphere. If \(0<\beta<1\), they move toward the equator.
The poles and the equator remain fixed.

%% ------------------------------------------------------------
%% Subsection: Harmonic mode coupling
%% ------------------------------------------------------------
\subsection{Harmonic mode coupling}
\label{subsec:no-order-mixing}

We now expand a transformed spherical harmonic in the
original basis. The following theorem gives an integral formula for the expansion coefficients.
We use them later to express the Jones blocks of the kernel.

\begin{theorem}[Connection coefficients]
\label{thm:no-mixing-S2}
For every \(\ell\ge0\) and \(|m|\le\ell\), the pulled-back harmonic has the
\(L^2(\Sph^2)\)-expansion
\begin{equation}
\label{eq:pullback-harmonic-S2}
 Y_{\ell m}\bigl(T_\beta(s)\bigr)
 =
 \sum_{n\ge |m|}
 B_n^{(\ell,m)}(\beta)Y_{nm}(s).
\end{equation}
No terms of order \(m'\neq m\) occur. With
\(\mu=|m|\) and \(\widetilde x_\beta\) as defined in
\eqref{eq:xtilde-def}, the connection coefficients are
\begin{equation}
\label{eq:B-integral-S2}
 B_n^{(\ell,m)}(\beta)
 =
 2\pi\mathcal N_{\ell\mu}\mathcal N_{n\mu}
 \int_{-1}^{1}
 P_\ell^\mu\bigl(\widetilde x_\beta(x)\bigr)
 P_n^\mu(x)\,\dd x.
\end{equation}
\end{theorem}

\begin{proof}
{
Write \(s=\omega(L,\lambda)\). For \(m\ge0\),
\eqref{eq:Tbeta-canonical-coordinate} gives
\begin{equation}
\label{eq:transformed-harmonic-coordinate}
 Y_{\ell m}\bigl(T_\beta(\omega(L,\lambda))\bigr)
 =
 \mathcal N_{\ell m}
 P_\ell^m\bigl(\widetilde x_\beta(\cos L)\bigr)
 e^{im\lambda}.
\end{equation}
Thus the longitude dependence is \(e^{im\lambda}\). For negative \(m\),
the same conclusion follows from the convention
\eqref{eq:Ylm-negative}.

Since \(Y_{\ell m}\circ T_\beta\in L^2(\Sph^2)\), it has the expansion
\[
 Y_{\ell m}\circ T_\beta
 =
 \sum_{n=0}^{\infty}
 \sum_{m'=-n}^{n}
 \left\langle
 Y_{\ell m}\circ T_\beta,
 Y_{nm'}
 \right\rangle_{L^2(\Sph^2)}
 Y_{nm'}.
\]
{
The integral over \(\lambda\) in each coefficient vanishes unless
\(m'=m\), since
\[
\int_0^{2\pi}
e^{i(m-m')\lambda}\,\dd\lambda
=
2\pi\delta_{mm'}.
\]
Hence only terms with \(m'=m\) remain, and necessarily \(n\ge|m|\),
which gives \eqref{eq:pullback-harmonic-S2}.
}

For \(m\ge0\), substituting the spherical-harmonic representations gives
\[
\begin{aligned}
 B_n^{(\ell,m)}(\beta)
 &=
 \int_0^{2\pi}\int_0^\pi
 Y_{\ell m}\bigl(\omega(\Theta_\beta(L),\lambda)\bigr)
 \overline{Y_{nm}\bigl(\omega(L,\lambda)\bigr)}
 \sin L\,\dd L\,\dd\lambda
 \\
 &=
 2\pi\mathcal N_{\ell m}\mathcal N_{nm}
 \int_0^\pi
 P_\ell^m\bigl(\cos\Theta_\beta(L)\bigr)
 P_n^m(\cos L)
 \sin L\,\dd L.
\end{aligned}
\]
Since
\[
 \cos\Theta_\beta(L)=\widetilde x_\beta(\cos L),
\]
the substitution
\[
 x=\cos L,
 \qquad
 \dd x=-\sin L\,\dd L,
\]
gives \eqref{eq:B-integral-S2}.
}
\end{proof}

%% ------------------------------------------------------------
%% Subsection: Jones blocks of the deformed kernel
%% ------------------------------------------------------------
\subsection{Jones blocks of the deformed kernel}
\label{subsec:pullback-jones-coefficients}

We use the connection coefficients from the previous
theorem and the isotropic coefficients \(\kap_\ell\) to derive a formula
for the entries of the Jones blocks of \(K_\beta\).

\begin{theorem}[Jones coefficients]
\label{thm:coefficient-jones-blocks}
Let \(K_{\iso}\) be the isotropic kernel in
\eqref{eq:Kiso-expansion-S2}, where \(\kap_\ell\ge0\) and
\eqref{eq:schoenberg-summability-kappa} holds, and let \(K_\beta\) be
defined by \eqref{eq:Kbeta-def}. Then
\[
 a_{nn'}^{(m,m')}(\beta)=0
 \qquad\text{whenever }m\neq m'.
\]
For \(m=m'\), write
\[
 a_{nn'}^{(m)}(\beta)
 :=
 a_{nn'}^{(m,m)}(\beta).
\]
The kernel consequently has the Jones-block expansion
\begin{equation}
\label{eq:pullback-jones-block-expansion}
 K_\beta(s,t)
 =
 \sum_{m\in\Z}
 \sum_{n,n'\ge |m|}
 a_{nn'}^{(m)}(\beta)
 Y_{nm}(s)
 \overline{Y_{n'm}(t)},
\end{equation}
with convergence in \(L^2(\Sph^2\times\Sph^2)\).
Its block coefficients are given by
\begin{equation}
\label{eq:jones-block-coefficients}
 a_{nn'}^{(m)}(\beta)
 =
 \sum_{\ell\ge |m|}
 \kap_\ell
 B_n^{(\ell,m)}(\beta)
 \overline{B_{n'}^{(\ell,m)}(\beta)}.
\end{equation}
\end{theorem}

\begin{proof}
{
Substituting \eqref{eq:Kbeta-composed-series} into the definition
\eqref{eq:canonical-pullback-coefficients} and reordering yields 
\begin{equation*}
a_{nn'}^{(m,m')}(\beta)
:= \sum_{\ell=0}^{\infty}\kap_\ell \sum_{r=-\ell}^{\ell}
\int_{\Sph^2}\int_{\Sph^2}
Y_{\ell r}\bigl(T_\beta(s)\bigr)
\overline{Y_{\ell r}\bigl(T_\beta(t)\bigr)}
\overline{Y_{nm}(s)}
Y_{n'm'}(t)
\,\dd\omega(t)\,\dd\omega(s).
\end{equation*}
For each \(\ell\) and \(r\),
the integral with respect to \(s\) is
\[
 \int_{\Sph^2}
 Y_{\ell r}\bigl(T_\beta(s)\bigr)
 \overline{Y_{nm}(s)}
 \,\dd\omega(s)
 =
 B_n^{(\ell,r)}(\beta)\delta_{rm},
\]
by Theorem~\ref{thm:no-mixing-S2}. Similarly, the integral with respect
to \(t\) is
\[
 \int_{\Sph^2}
 \overline{Y_{\ell r}\bigl(T_\beta(t)\bigr)}
 Y_{n'm'}(t)
 \,\dd\omega(t)
 =
 \overline{B_{n'}^{(\ell,r)}(\beta)}\delta_{rm'}.
\]
It follows that
\[
 a_{nn'}^{(m,m')}(\beta)
 =
 \sum_{\ell=0}^{\infty}\kap_\ell
 \sum_{r=-\ell}^{\ell}
 B_n^{(\ell,r)}(\beta)
 \overline{B_{n'}^{(\ell,r)}(\beta)}
 \delta_{rm}\delta_{rm'}.
\]
Hence \(a_{nn'}^{(m,m')}(\beta)=0\) for \(m\neq m'\), while for
\(m=m'\),
\[
 a_{nn'}^{(m)}(\beta)
 =
 \sum_{\ell\ge|m|}
 \kap_\ell
 B_n^{(\ell,m)}(\beta)
 \overline{B_{n'}^{(\ell,m)}(\beta)}.
\]
This proves \eqref{eq:jones-block-coefficients}.

It remains to justify the series in
\eqref{eq:jones-block-coefficients}. By the definition of the connection
coefficients,
\[
 \bigl|B_n^{(\ell,m)}(\beta)\bigr|
 \le
 \|Y_{\ell m}\circ T_\beta\|_{L^\infty(\Sph^2)}
 \|Y_{nm}\|_{L^1(\Sph^2)}.
\]
The addition theorem and Cauchy--Schwarz give
\[
 \|Y_{\ell m}\circ T_\beta\|_{L^\infty(\Sph^2)}
 \le
 \left(\frac{2\ell+1}{4\pi}\right)^{1/2},
 \qquad
 \|Y_{nm}\|_{L^1(\Sph^2)}
 \le
 (4\pi)^{1/2}.
\]
Thus
\[
 \bigl|B_n^{(\ell,m)}(\beta)\bigr|
 \le
 (2\ell+1)^{1/2},
\]
and therefore
\[
 \sum_{\ell\ge|m|}
 \kap_\ell
 \bigl|B_n^{(\ell,m)}(\beta)\bigr|
 \bigl|B_{n'}^{(\ell,m)}(\beta)\bigr|
 \le
 \sum_{\ell\ge|m|}
 (2\ell+1)\kap_\ell
 <\infty.
\]

Finally, \(K_\beta\) is continuous on the compact space
\(\Sph^2\times\Sph^2\), so it belongs to
\(L^2(\Sph^2\times\Sph^2)\). Its spherical-harmonic expansion therefore
converges in \(L^2\). Since all coefficients with \(m\neq m'\) vanish,
the general expansion \eqref{eq:Kbeta-general-harmonic-expansion}
reduces to \eqref{eq:pullback-jones-block-expansion}.
} 
\end{proof}

%% ------------------------------------------------------------
%% Subsection: Explicit formula for the connection coefficients
%% ------------------------------------------------------------
\subsection{Explicit formula for the connection coefficients}
\label{subsec:explicit-formula}

We use \eqref{eq:gamma-def-main} and Euler's integral representation
to write the connection coefficients as a finite sum involving
Beta and Gauss hypergeometric functions.

We recall the standard Beta-function identity
\begin{equation}
\label{eq:Euler-beta-function}
 \Bfun(a,b)
 =
 \int_0^1 t^{a-1}(1-t)^{b-1}\,\dd t
 =
 \frac{\Gamma(a)\Gamma(b)}{\Gamma(a+b)},
 \qquad a,b>0.
\end{equation}
The Gauss hypergeometric function is initially defined by
\begin{equation}
\label{eq:Gauss-hypergeometric-series}
 {}_2F_1(a,b;c;z)
 =
 \sum_{r=0}^{\infty}
 \frac{(a)_r(b)_r}{(c)_r\,r!}z^r,
 \qquad |z|<1,
\end{equation}
and extended by analytic continuation. For
\(c>b>0\) and real \(z<1\), Euler's integral representation is
\begin{equation}
\label{eq:Euler-hypergeometric-integral}
 \Bfun(b,c-b)\,{}_2F_1(a,b;c;z)
 =
 \int_0^1
 t^{b-1}(1-t)^{c-b-1}(1-zt)^{-a}\,\dd t;
\end{equation}
see \cite[Eq.~15.6.1]{NIST:DLMF}. In our application,
\(z=1-\beta^2<1\), so \eqref{eq:Euler-hypergeometric-integral} applies for
every \(\beta>0\).

\begin{theorem}[Explicit connection coefficients]
\label{thm:explicit-S2}
Let \(\ell\ge0\), \(|m|\le\ell\), and \(n\ge|m|\), and set
\(\mu=|m|\). If \(\ell+n\) is odd, then
\(B_n^{(\ell,m)}(\beta)=0\). If \(\ell+n\) is even, then
\begin{align}
\label{eq:explicit-B-S2}
 B_n^{(\ell,m)}(\beta)
 &=
 2\pi\mathcal N_{\ell\mu}\mathcal N_{n\mu}
 \sum_{k=0}^{R_{\ell,\mu}}
 \sum_{j=0}^{R_{n,\mu}}
 \gamma_{\ell,\mu,k}\gamma_{n,\mu,j}
 \beta^{\ell-\mu-2k}
 \notag\\
 &\quad\times
 \Bfun\left(p_{k,j}+\frac12,\mu+1\right)
 {}_2F_1\left(
 s_k,p_{k,j}+\frac12;
 p_{k,j}+\mu+\frac32;
 1-\beta^2
 \right),
\end{align}
where \(R_{a,\mu}\) and
\(\gamma_{a,\mu,k}\), for \(a\in\{\ell,n\}\), are defined in
\eqref{eq:gamma-def-main}, and
\begin{equation}
\label{eq:sk-pkj-S2}
 s_k=\frac{\ell-2k}{2},
 \qquad
 p_{k,j}
 =
 \frac{\ell+n-2\mu-2(k+j)}{2}
 \in\mathbb N_0.
\end{equation}
\end{theorem}

\begin{proof}
See Appendix~\ref{app:explicit-S2}. It proves the parity rule, derives the
finite associated-Legendre expansions, and evaluates the remaining integrals using \eqref{eq:Euler-hypergeometric-integral}.
\end{proof}

%% ------------------------------------------------------------
%% Subsection: Geometric decay for a fixed input mode
%% ------------------------------------------------------------
\subsection{Geometric decay for a fixed input mode}
\label{subsec:fixed-mode-decay}

We now fix an input mode \((\ell,m)\) and study
\(B_n^{(\ell,m)}(\beta)\) as \(n\to\infty\). We show that the
coefficients decay geometrically, with constants that may depend
on \((\ell,m)\). However, this estimate alone does not give a
uniform bound for an entire Jones block.

Fix \(m\in\Z\), set \(\mu=|m|\), and let \(\ell\ge\mu\). If
\(\beta=1\), then \(T_\beta\) is the identity and
\(B_n^{(\ell,m)}(1)=\delta_{n\ell}\). We therefore assume
\(\beta\neq1\).

By \eqref{eq:associated-Legendre-finite}, together with
\[
 \widetilde x_\beta(x)
 =
 \frac{\beta x}{\sqrt{1+(\beta^2-1)x^2}},
 \qquad
 1-\widetilde x_\beta(x)^2
 =
 \frac{1-x^2}{1+(\beta^2-1)x^2},
\]
we obtain
\begin{equation}
\label{eq:H-factor-main}
 P_\ell^\mu\bigl(\widetilde x_\beta(x)\bigr)
 =
 (1-x^2)^{\mu/2}H_{\ell,\mu,\beta}(x),
\end{equation}
where
\begin{equation}
\label{eq:H-def-main}
 H_{\ell,\mu,\beta}(x)
 =
 \sum_{k=0}^{R_{\ell,\mu}}
 \gamma_{\ell,\mu,k}
 \beta^{\ell-\mu-2k}
 x^{\ell-\mu-2k}
 \bigl(1+(\beta^2-1)x^2\bigr)^{-(\ell-2k)/2}.
\end{equation}

By \eqref{eq:H-factor-main}, the integral in
\eqref{eq:B-integral-S2} becomes
\[
\int_{-1}^{1}
H_{\ell,\mu,\beta}(x)
P_n^\mu(x)
(1-x^2)^{\mu/2}\,\dd x.
\]
{
Using the relation between associated Legendre and Gegenbauer polynomials,
this integral is, up to explicit normalization factors, a Gegenbauer
coefficient of \(H_{\ell,\mu,\beta}\) of degree \(n-\mu\) and parameter
\(\mu+\frac12\). Therefore we can employ the geometric coefficient estimate introduced in Wang~\cite[Theorem~4.3]{Wang2016Gegenbauer} to bound our coefficients.

}

\begin{theorem}[Geometric decay for a fixed input mode]
\label{thm:fixed-mode-decay}
Fix \(m\in\Z\), set \(\mu=|m|\), and let \(\ell\ge\mu\). Suppose that
\(\beta>0\) and \(\beta\neq1\), and define
\begin{equation}
\label{eq:rho-beta-def}
 \rho_\beta
 =
 \sqrt{\frac{1+\beta}{|1-\beta|}},
 \qquad
 r_\beta
 =
 \rho_\beta^{-1}
 =
 \sqrt{\frac{|1-\beta|}{1+\beta}}.
\end{equation}
Then, for every \(1<\rho<\rho_\beta\), there exists
\(C_{\ell,m,\beta,\rho}>0\) such that
\begin{equation}
\label{eq:fixed-mode-geometric-decay}
 \bigl|B_n^{(\ell,m)}(\beta)\bigr|
 \le
 C_{\ell,m,\beta,\rho}\rho^{-n},
 \qquad n\ge\mu.
\end{equation}
Consequently,
\begin{equation}
\label{eq:limsup-decay}
 \limsup_{n\to\infty}
 \bigl|B_n^{(\ell,m)}(\beta)\bigr|^{1/n}
 \le r_\beta.
\end{equation}
\end{theorem}

\begin{proof}
See Appendix~\ref{app:decay-S2}.
\end{proof}

For each fixed input mode, the connection coefficients
decay geometrically as the output degree increases. The upper bound \(r_\beta\) on the asymptotic geometric decay rate
tends to zero as \(\beta\) approaches \(1\), which is consistent with
the absence of degree coupling when \(\beta=1\).

%% ------------------------------------------------------------
%% Subsection: Native spaces and Sobolev regularity
%% ------------------------------------------------------------
\subsection{Native spaces and Sobolev regularity}
\label{subsec:sobolev-consequences}

Since \(T_\beta\) is a smooth diffeomorphism, the general transport results
from Section~\ref{sec:general-pullbacks} apply directly.

\begin{corollary}[Native space of the pullback kernel]
\label{cor:nativespace-canonical-pullback}
Assume the hypotheses of
Proposition~\ref{prop:isotropic-native-space-spectral}. Then
{
\begin{equation}
\label{eq:canonical-native-space-equivalence}
 \Nspace(K_\beta)\equiv H^\tau(\Sph^2).
\end{equation}
}
\end{corollary}

\begin{proof}
Proposition~\ref{prop:isotropic-native-space-spectral} gives
\[
 \Nspace(K_{\iso})\equiv H^\tau(\Sph^2).
\]
Since \(T_\beta\) is a \(C^\infty\)-diffeomorphism, the result follows from
Corollary~\ref{cor:native-transport} with \(\Phi=T_\beta\).
\end{proof}

The native space has the same Sobolev order before
and after deformation, even though the covariance structure and
the harmonic coefficients can change. We now use Sobolev duality
to derive estimates for the Jones blocks.

\begin{corollary}
\label{cor:sobolev-jones-bound}
Assume \eqref{eq:two-sided-spectral-assumption} for some \(\tau>1\), and let
\(\beta>0\). Then there exist constants
\(c_{\beta,\tau},C_{\beta,\tau}>0\), independent of \(m\) and the harmonic
degrees, such that the following holds.

Fix \(m\in\Z\), set \(\mu=|m|\), and let \(N\ge\mu\).
Then, for every choice of coefficients
\(\widehat f_{nm}\), \(n=\mu,\ldots,N\),
\begin{align}
\label{eq:jones-block-two-sided-sobolev-bound}
c_{\beta,\tau}c_-
\sum_{n=\mu}^{N}
(1+\Lambda_n)^{-\tau}
\bigl|\widehat f_{nm}\bigr|^2
&\le
\sum_{n,n'=\mu}^{N}
a_{nn'}^{(m)}(\beta)
\widehat f_{n'm}
\overline{\widehat f_{nm}}
\notag\\
&\le
C_{\beta,\tau}c_+
\sum_{n=\mu}^{N}
(1+\Lambda_n)^{-\tau}
\bigl|\widehat f_{nm}\bigr|^2.
\end{align}
In particular,
\begin{equation}
\label{eq:jones-diagonal-two-sided-bound}
c_{\beta,\tau}c_-
(1+\Lambda_n)^{-\tau}
\le
a_{nn}^{(m)}(\beta)
\le
C_{\beta,\tau}c_+
(1+\Lambda_n)^{-\tau},
\qquad n\ge|m|,
\end{equation}
and
\begin{equation}
\label{eq:jones-off-diagonal-upper-bound}
\bigl|a_{nn'}^{(m)}(\beta)\bigr|
\le
C_{\beta,\tau}c_+
(1+\Lambda_n)^{-\tau/2}
(1+\Lambda_{n'})^{-\tau/2}.
\end{equation}
\end{corollary}

\begin{proof}
{
For the coefficients in the statement, define \(f\) by
\[
f
:=
\sum_{n=\mu}^{N}
\widehat f_{nm}Y_{nm}.
\]
By Corollary~\ref{cor:nativespace-canonical-pullback},
Remark~\ref{rem:dual-norm-comparison}, and the reproducing property,
there exist constants \(c_{\beta,\tau},C_{\beta,\tau}>0\) such that
\begin{equation}
\label{eq:pullback-kernel-dual-sobolev-bound}
\begin{aligned}
c_{\beta,\tau}c_-
\lVert f\rVert_{H^{-\tau}(\Sph^2)}^2
&\le
\lVert f\rVert_{\mathcal N(K_\beta)^*}^2
\\
&=
\int_{\Sph^2}\int_{\Sph^2}
K_\beta(s,t)f(t)\overline{f(s)}
\,\dd\omega(t)\,\dd\omega(s)
\\
&\le
C_{\beta,\tau}c_+
\lVert f\rVert_{H^{-\tau}(\Sph^2)}^2.
\end{aligned}
\end{equation}
}

{
Expanding \(f(t)\) and \(\overline{f(s)}\), and using the definition of the
Jones coefficients, gives
\[
\begin{aligned}
&\int_{\Sph^2}\int_{\Sph^2}
K_\beta(s,t)f(t)\overline{f(s)}
\,\dd\omega(t)\,\dd\omega(s)
\\
&=
\int_{\Sph^2}\int_{\Sph^2}
K_\beta(s,t)
\left(
\sum_{n'=\mu}^{N}
\widehat f_{n'm}Y_{n'm}(t)
\right)
\left(
\sum_{n=\mu}^{N}
\overline{\widehat f_{nm}}\,
\overline{Y_{nm}(s)}
\right)
\,\dd\omega(t)\,\dd\omega(s)
\\
&=
\sum_{n,n'=\mu}^{N}
\widehat f_{n'm}
\overline{\widehat f_{nm}}
\int_{\Sph^2}\int_{\Sph^2}
K_\beta(s,t)
\overline{Y_{nm}(s)}
Y_{n'm}(t)
\,\dd\omega(t)\,\dd\omega(s)
\\
&=
\sum_{n,n'=\mu}^{N}
a_{nn'}^{(m)}(\beta)
\widehat f_{n'm}
\overline{\widehat f_{nm}}.
\end{aligned}
\]
}
Moreover,
\[
\lVert f\rVert_{H^{-\tau}(\Sph^2)}^2
=
\sum_{n=\mu}^{N}
(1+\Lambda_n)^{-\tau}
\bigl|\widehat f_{nm}\bigr|^2.
\]
Substitution into
\eqref{eq:pullback-kernel-dual-sobolev-bound} gives
\eqref{eq:jones-block-two-sided-sobolev-bound}. Taking \(f=Y_{nm}\) gives
\eqref{eq:jones-diagonal-two-sided-bound}.

Finally, \eqref{eq:jones-block-coefficients} and Cauchy--Schwarz give
\[
\begin{aligned}
\bigl|a_{nn'}^{(m)}(\beta)\bigr|^2
&\le
\left(
\sum_{\ell\ge|m|}
\kap_\ell
\bigl|B_n^{(\ell,m)}(\beta)\bigr|^2
\right)
\left(
\sum_{\ell\ge|m|}
\kap_\ell
\bigl|B_{n'}^{(\ell,m)}(\beta)\bigr|^2
\right)
\\
&=
a_{nn}^{(m)}(\beta)
a_{n'n'}^{(m)}(\beta).
\end{aligned}
\]
Applying the upper diagonal bound to both factors gives
\eqref{eq:jones-off-diagonal-upper-bound}.
\end{proof}

%% ------------------------------------------------------------
%% Subsection: Extension to an arbitrary axis
%% ------------------------------------------------------------
\subsection{Extension to an arbitrary axis}
\label{subsec:arbitrary-axis}

Let \(u\in\Sph^2\), choose \(Q_u\in\SO(3)\) with \(Q_ue_3=u\), and define
\begin{equation}
\label{eq:Aubeta-def}
 A_{u,\beta}
 =
 I_3+(\beta-1)uu^\top,
 \qquad
 T_{u,\beta}(s)
 =
 \frac{A_{u,\beta}s}{\|A_{u,\beta}s\|}.
\end{equation}
Then
\begin{equation}
\label{eq:conjugacy-S2}
 A_{u,\beta}=Q_uA_\beta Q_u^\top,
 \qquad
 T_{u,\beta}=Q_u\circ T_\beta\circ Q_u^\top.
\end{equation}
Hence \(T_{u,\beta}\) is a smooth diffeomorphism with inverse
\(T_{u,\beta^{-1}}\). Moreover,
\[
 A_{-u,\beta}=A_{u,\beta},
 \qquad
 T_{-u,\beta}=T_{u,\beta},
\]
so the deformation depends only on the axis \(\R u\), regardless of
its orientation.

To work with this axis, we define the rotated spherical harmonics by
\begin{equation}
\label{eq:Yu-def-framework}
 Y_{\ell m}^{u}(s)=Y_{\ell m}(Q_u^\top s).
\end{equation}
Using \eqref{eq:conjugacy-S2} and
\eqref{eq:pullback-harmonic-S2}, we obtain
\begin{align}
\label{eq:pullback-harmonic-arbitrary-axis}
 Y_{\ell m}^{u}\bigl(T_{u,\beta}(s)\bigr)
 &=
 Y_{\ell m}\bigl(T_\beta(Q_u^\top s)\bigr)
 \notag\\
 &=
 \sum_{n\ge|m|}
 B_n^{(\ell,m)}(\beta)Y_{nm}^{u}(s).
\end{align}
Thus rotating the axis changes the adapted basis
but not the connection coefficients.

For
\begin{equation}
\label{eq:Kubeta-def}
K_{u,\beta}(s,t)
=
K_{\iso}\bigl(T_{u,\beta}(s),T_{u,\beta}(t)\bigr),
\end{equation}
let \(R\in\SO(3)\) satisfy \(Ru=u\). Then \(Q_u^\top RQ_u\) fixes
\(e_3\), so the equivariance of \(T_\beta\) gives
\[
T_{u,\beta}(Rs)=R\,T_{u,\beta}(s).
\]
By isotropy of \(K_{\iso}\),
\[
K_{u,\beta}(Rs,Rt)=K_{u,\beta}(s,t).
\]
Hence \(K_{u,\beta}\) is axially symmetric about \(u\).

Rotations preserve surface measure and act isometrically
on \(L^2(\Sph^2)\) and every \(H^\tau(\Sph^2)\). Therefore the results
for the connection coefficients, Jones coefficients, native spaces,
and Sobolev norms transfer to \(T_{u,\beta}\) and \(K_{u,\beta}\)
in the axis-adapted basis. The Sobolev constants can be chosen
independently of \(u\).

A different choice of \(Q_u\) multiplies the adapted
harmonics of a fixed order \(m\) by a common phase factor.
This factor cancels in the definitions of the connection and Jones
coefficients, so these coefficients do not depend on the choice
of \(Q_u\).
%% ------------------------------------------------------------
%% Subsection: Interpolation and error estimates
%% ------------------------------------------------------------

{
\subsection{Interpolation and error estimates}
\label{subsec:interpolation}

Interpolation with \(K_{u,\beta}\) is equivalent to
interpolation with \(K_{\iso}\) at a transformed point set. We use this
relation to transfer the known isotropic error estimates to
\(K_{u,\beta}\).

Assume throughout that \(K_{\iso}\) satisfies the coefficient decay hypotheses of
Proposition~\ref{prop:isotropic-native-space-spectral} for some
\(\tau>1\). Let
\[
X=\{x_1,\ldots,x_N\}\subset\Sph^2
\]
be a set of pairwise distinct points, and define
\[
X_{u,\beta}
:=
T_{u,\beta}(X)
=
\{T_{u,\beta}(x_1),\ldots,T_{u,\beta}(x_N)\}.
\]
Denote by \(I_X^{u,\beta}f\) the interpolant of \(f\) on \(X\) with
kernel \(K_{u,\beta}\), and by \(I_Z^{\iso}g\) the interpolant of \(g\)
on a finite set \(Z\subset\Sph^2\) with kernel \(K_{\iso}\).

\begin{proposition}
\label{prop:interpolation-transport}
Let
\[
g:=f\circ T_{u,\beta}^{-1}.
\]
Then
\begin{equation}
\label{eq:interpolation-transport}
I_X^{u,\beta}f
=
\bigl(I_{X_{u,\beta}}^{\iso}g\bigr)\circ T_{u,\beta}.
\end{equation}
\end{proposition}

\begin{proof}

The two interpolation systems have the same matrix and data vector,
so they give the same coefficients. Substituting the definition of
\(K_{u,\beta}\) into the interpolant then gives
\eqref{eq:interpolation-transport}.

\end{proof}

For a finite set
\(Z\subset\Sph^2\), let
\[
h_Z
:=
\sup_{s\in\Sph^2}
\min_{z\in Z}d_{\Sph^2}(s,z)
\]
denote its geodesic fill distance.

Set
\[
L_\beta:=\max\{\beta,\beta^{-1}\}.
\]
For \(v\in T_s\Sph^2\), differentiation of \(T_\beta\) gives
\[
D T_\beta(s)v
=
\frac{
\bigl(I-T_\beta(s)T_\beta(s)^\top\bigr)A_\beta v
}{
\lVert A_\beta s\rVert
}.
\]
Since the projection in the numerator has operator norm one,
\[
\lVert D T_\beta(s)v\rVert
\le
\frac{\max\{1,\beta\}}{\min\{1,\beta\}}\lVert v\rVert
=
L_\beta\lVert v\rVert.
\]
By \eqref{eq:conjugacy-S2}, the same derivative bound
holds for \(T_{u,\beta}\). For \(s,t\in\Sph^2\), let
\(\gamma:[0,1]\to\Sph^2\) be a minimizing geodesic from \(s\) to \(t\).
Then
\[
\begin{aligned}
d_{\Sph^2}\bigl(T_{u,\beta}(s),T_{u,\beta}(t)\bigr)
&\le
\operatorname{length}(T_{u,\beta}\circ\gamma)
\\
&=
\int_0^1
\bigl\lVert
D T_{u,\beta}\bigl(\gamma(r)\bigr)\gamma'(r)
\bigr\rVert
\,\dd r
\\
&\le
L_\beta\int_0^1\lVert\gamma'(r)\rVert\,\dd r
\\
&=
L_\beta d_{\Sph^2}(s,t).
\end{aligned}
\]
Thus \(T_{u,\beta}\) is \(L_\beta\)-Lipschitz with respect to the
geodesic distance. Using the bijectivity of \(T_{u,\beta}\), we obtain
\begin{equation}
\label{eq:fill-distance-deformation}
\begin{aligned}
h_{X_{u,\beta}}
&=
\sup_{s\in\Sph^2}
\min_{x\in X}
d_{\Sph^2}\bigl(T_{u,\beta}(s),T_{u,\beta}(x)\bigr)\le
L_\beta h_X.
\end{aligned}
\end{equation}

By \eqref{eq:two-sided-spectral-assumption}, the kernel \(K_{\iso}\)
satisfies the coefficient condition of Narcowich--Sun--Ward--Wendland
\cite[Theorem~5.5]{NarcowichSunWardWendland2007}. Applying their result with
both the kernel smoothness and the target smoothness equal to \(\tau\) gives
\begin{equation}
\label{eq:isotropic-error-estimate}
\bigl\lVert g-I_Z^{\iso}g\bigr\rVert_{H^\sigma(\Sph^2)}
\le
C_{\iso,\tau,\sigma}
h_Z^{\tau-\sigma}
\lVert g\rVert_{H^\tau(\Sph^2)},
\qquad
0\le\sigma\le\tau,
\end{equation}
for every \(g\in H^\tau(\Sph^2)\) and every finite set
\(Z\subset\Sph^2\) of pairwise distinct points.

\begin{corollary}
\label{cor:interpolation-error}
For every \(f\in H^\tau(\Sph^2)\) and \(0\le\sigma\le\tau\), there
exists a constant \(C_{\beta,\tau,\sigma}>0\), independent of \(u\),
\(X\), and \(f\), such that
\begin{equation}
\label{eq:deformed-error-estimate}
\begin{aligned}
\bigl\lVert f-I_X^{u,\beta}f\bigr\rVert_{H^\sigma(\Sph^2)}
&\le
C_{\beta,\tau,\sigma}
h_{X_{u,\beta}}^{\tau-\sigma}
\lVert f\rVert_{H^\tau(\Sph^2)}
\\
&\le
C_{\beta,\tau,\sigma}
L_\beta^{\tau-\sigma}
h_X^{\tau-\sigma}
\lVert f\rVert_{H^\tau(\Sph^2)}.
\end{aligned}
\end{equation}
\end{corollary}

\begin{proof}
Set \(T=T_{u,\beta}\) and \(g=f\circ T^{-1}\). By
Proposition~\ref{prop:interpolation-transport},
\[
f-I_X^{u,\beta}f
=
\bigl(g-I_{X_{u,\beta}}^{\iso}g\bigr)\circ T.
\]
Proposition~\ref{prop:sobolev-stability-diffeomorphism}, applied first
to \(T\) and then to \(T^{-1}\), together with
\eqref{eq:isotropic-error-estimate}, gives
\[
\begin{aligned}
\bigl\lVert f-I_X^{u,\beta}f\bigr\rVert_{H^\sigma(\Sph^2)}
&\le
C_{\beta,\sigma}
\bigl\lVert
g-I_{X_{u,\beta}}^{\iso}g
\bigr\rVert_{H^\sigma(\Sph^2)}
\\
&\le
C_{\beta,\sigma}C_{\iso,\tau,\sigma}
h_{X_{u,\beta}}^{\tau-\sigma}
\lVert g\rVert_{H^\tau(\Sph^2)}
\\
&\le
C_{\beta,\tau,\sigma}
h_{X_{u,\beta}}^{\tau-\sigma}
\lVert f\rVert_{H^\tau(\Sph^2)}.
\end{aligned}
\]
The constants can be chosen independently of \(u\), since rotations act
isometrically on every \(H^s(\Sph^2)\). The second inequality in
\eqref{eq:deformed-error-estimate} follows from
\eqref{eq:fill-distance-deformation}.
\end{proof}

These estimates relate the original and transformed fill distances
and give the same convergence order as in the isotropic case.
They do not, however, guarantee a smaller interpolation error. To demonstrate in which situation a transformed kernel can outperform the isotropic one, we have to adopt a stochastic point of view and will do so in the following numerical experiments. 
}

%% ============================================================
%% Section: Numerical experiments
%% ============================================================

\section{Numerical experiments}
\label{sec:numerics}

{
We now study how the deformation changes correlations on the sphere when the kernel is used as a spherical covariance function. We also examine whether the deformed model improves predictions compared to an
isotropic model, in case the assumed covariance is not isotropic. Further, we examine how the models perform when the
fitted radial family is misspecified or the data are generated
using a different axial deformation.
}

%% ------------------------------------------------------------
%% Subsection: Simulation setup
%% ------------------------------------------------------------
\subsection{Simulation setup}

{
For each scenario, we run \(100\) independent repetitions unless
stated otherwise. Each time, we draw \(180\) locations uniformly
on the unit sphere \(\mathbb S^2\). We use \(120\) locations for
fitting and \(60\) for testing. The models use the same training
data, and we compare their predictions with the same test values.
We do not, however, optimize the sampling locations for the deformation.

The exceptions are Figures~\ref{fig:numerics-geometry}
and~\ref{fig:mid-spatial}, which are not based on these \(100\)
repetitions and are included mainly for illustration.
}

{
For a radial family \(F\), let \(C_F(r;\rho_F)\) be its correlation
function, where \(r\) is the chordal distance and \(\rho_F\) is the
range or scale parameter. We normalize it such that
\[
C_F(0;\rho_F)=1.
\]
With variance \(\sigma^2\), the isotropic covariance is
\[
K_F(s,t)
=
\sigma^2 C_F(\lVert s-t\rVert;\rho_F).
\]
For our axial model, we use
\[
A_{u,\beta}=I+(\beta-1)uu^\top,
\qquad
T_{u,\beta}(s)
=
\frac{A_{u,\beta}s}{\lVert A_{u,\beta}s\rVert}
\]
and define the covariance by
\[
K_{u,\beta,F}(s,t)
=
\sigma^2
C_F\!\left(
\lVert T_{u,\beta}(s)-T_{u,\beta}(t)\rVert;\rho_F
\right).
\]
For \(\beta=1\), the map is the identity.
}

{
For each repetition, let \(K_0\) be the covariance kernel used to
generate the data. We set \(\sigma_0^2=1\) and use no nugget, so
\(K_0(s,s)=1\). Conditional on the sampled locations
\(s_1,\ldots,s_N\), we use the Gaussian model
\[
\begin{aligned}
\mathbf Z
&:=
\bigl(Z(s_1),\ldots,Z(s_N)\bigr)^\top,
\\
\mathbf Z\mid(s_1,\ldots,s_N)
&\sim
\mathcal N_N(0,\Sigma_0),
\qquad
(\Sigma_0)_{ij}=K_0(s_i,s_j).
\end{aligned}
\]
Each field value \(Z(s_i)\) therefore has marginal distribution
\(\mathcal N(0,1)\). Training and test values are generated jointly.
}

{
All fitted models have zero mean and no nugget. We estimate the
variance \(\sigma^2\) and the radial scale parameter in every model.
For our axial model, we also estimate the deformation strength
\(\beta\) and the unoriented axis \(u\). All free parameters are
estimated by Gaussian maximum likelihood.

We generate the data with \(u_0=e_3\), but treat the axis as unknown
in every axial fit. The choice \(u_0=e_3\) only fixes the coordinates,
since both the isotropic base covariance and the uniform sampling
design are rotationally invariant; see
Subsection~\ref{subsec:arbitrary-axis}.
}

{
For realization \(r\), we define the relative RMSE gain of a
fitted model \(M\) over the isotropic fit by
\[
G_{M,r}
=
100\,
\frac{\operatorname{RMSE}_{\mathrm{iso},r}
-\operatorname{RMSE}_{M,r}}
{\operatorname{RMSE}_{\mathrm{iso},r}}.
\]
Positive values mean that model \(M\) has a smaller RMSE than
the isotropic model. The win rate is the fraction of the
\(100\) repetitions with \(G_{M,r}>0\).
}

%% ------------------------------------------------------------
%% Subsection: Correlation geometry and radial families
%% ------------------------------------------------------------
\subsection{Correlation geometry and radial families}

{
In Figure~\ref{fig:numerics-geometry}, we first look at how the
deformation changes correlations on the sphere.
}

\begin{figure}[!htbp]
\centering
\includegraphics[height=0.35\textheight]{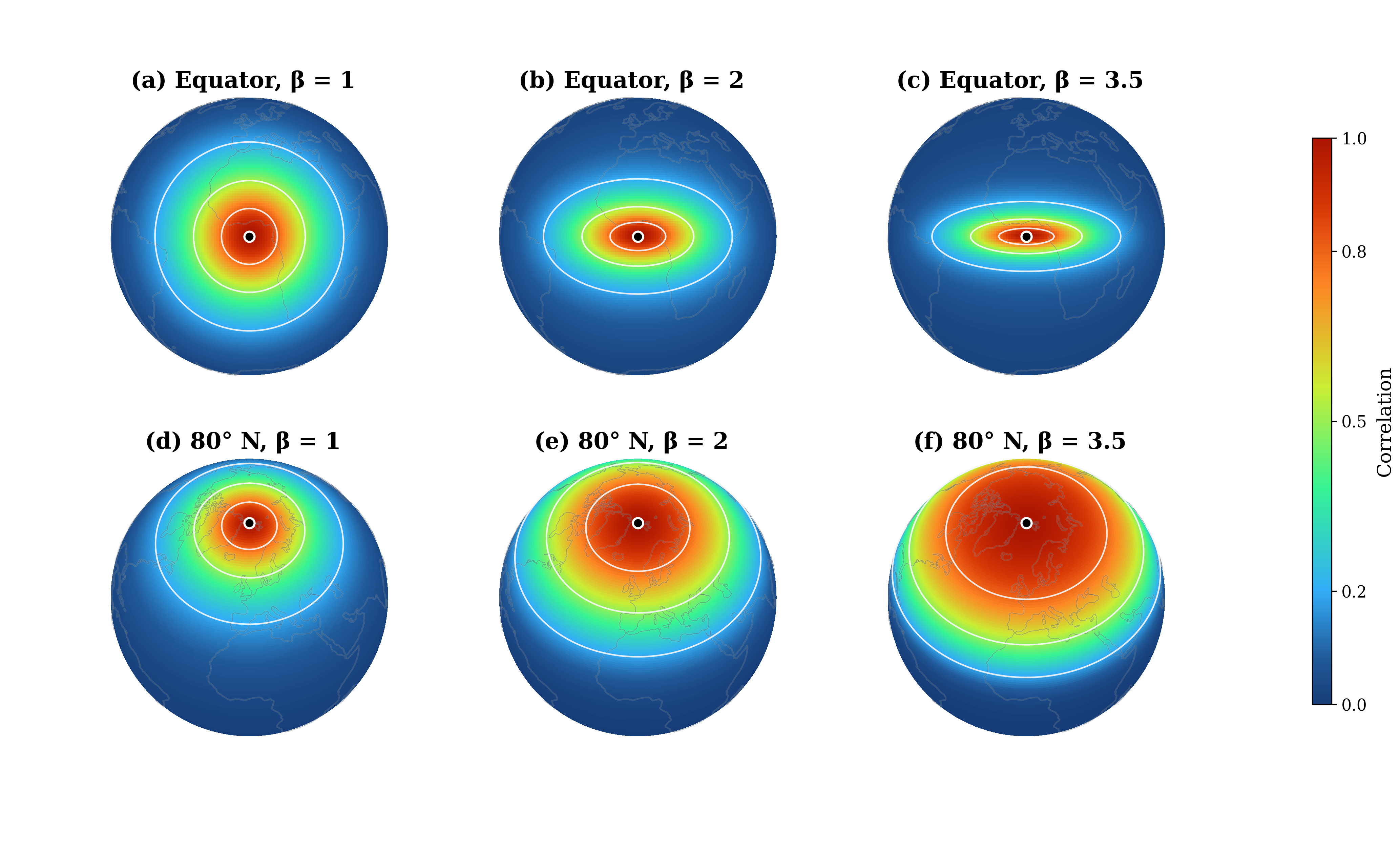}
\caption{Axial Mat\'ern correlation for source points
on the equator and at \(80^\circ\) N. Columns correspond to
\(\beta=1\), \(2\), and \(3.5\). The white curves show the correlation
levels \(0.2\), \(0.5\), and \(0.8\).}
\label{fig:numerics-geometry}
\end{figure}

{
For larger \(\beta\), the contours at the equator become narrower
from north to south and thus more elongated from east to west.
Near the pole, the region of high correlation grows. We get both
effects by changing \(\beta\) while keeping the axis fixed.
In contrast, an isotropic model cannot produce these differences
by changing its range.
}

{
We use three radial correlation families:
\[
\begin{array}{r@{\;:\qquad}l@{\;}c@{\;}l}
\text{Mat\'ern}
&
C_M(r;\rho)
&=&
\left(1+\dfrac{\sqrt3\,r}{\rho}\right)
\exp\left(-\dfrac{\sqrt3\,r}{\rho}\right),
\\[2mm]
\text{Gaussian}
&
C_G(r;\rho)
&=&
\exp\left(-\dfrac12\left(\dfrac r\rho\right)^2\right),
\\[2mm]
\text{Wendland}
&
C_W(r;R)
&=&
(1-r/R)_+^4(1+4r/R).
\end{array}
\]
The Mat\'ern smoothness is fixed at \(\nu=3/2\), and \(R\) is the
support radius of the Wendland function.

We choose the baseline values
\[
\rho_M=0.4,
\qquad
\rho_G=0.4,
\qquad
R_W=1.5
\]
so that the three isotropic models have approximately the same
mean correlation on \(\mathbb S^2\).
}

%% ------------------------------------------------------------
%% Subsection: Effect of deformation strength
%% ------------------------------------------------------------
\subsection{Effect of deformation strength}
\label{subsec:axial-strength}

{
We now generate data from our axial Mat\'ern model with
\[
\beta_0\in\{0.5,1,1.5,2.5,3.5\},
\qquad
\rho_0=0.4.
\]
In Figure~\ref{fig:axial-strength}, we compare an isotropic
Mat\'ern fit with our axial Mat\'ern fit. The axial fit uses
the same covariance family as the generating model, but the
axis and deformation strength must still be estimated from
the \(120\) training observations.
}

\begin{figure}[!htbp]
\centering
\includegraphics[height=0.35\textheight]{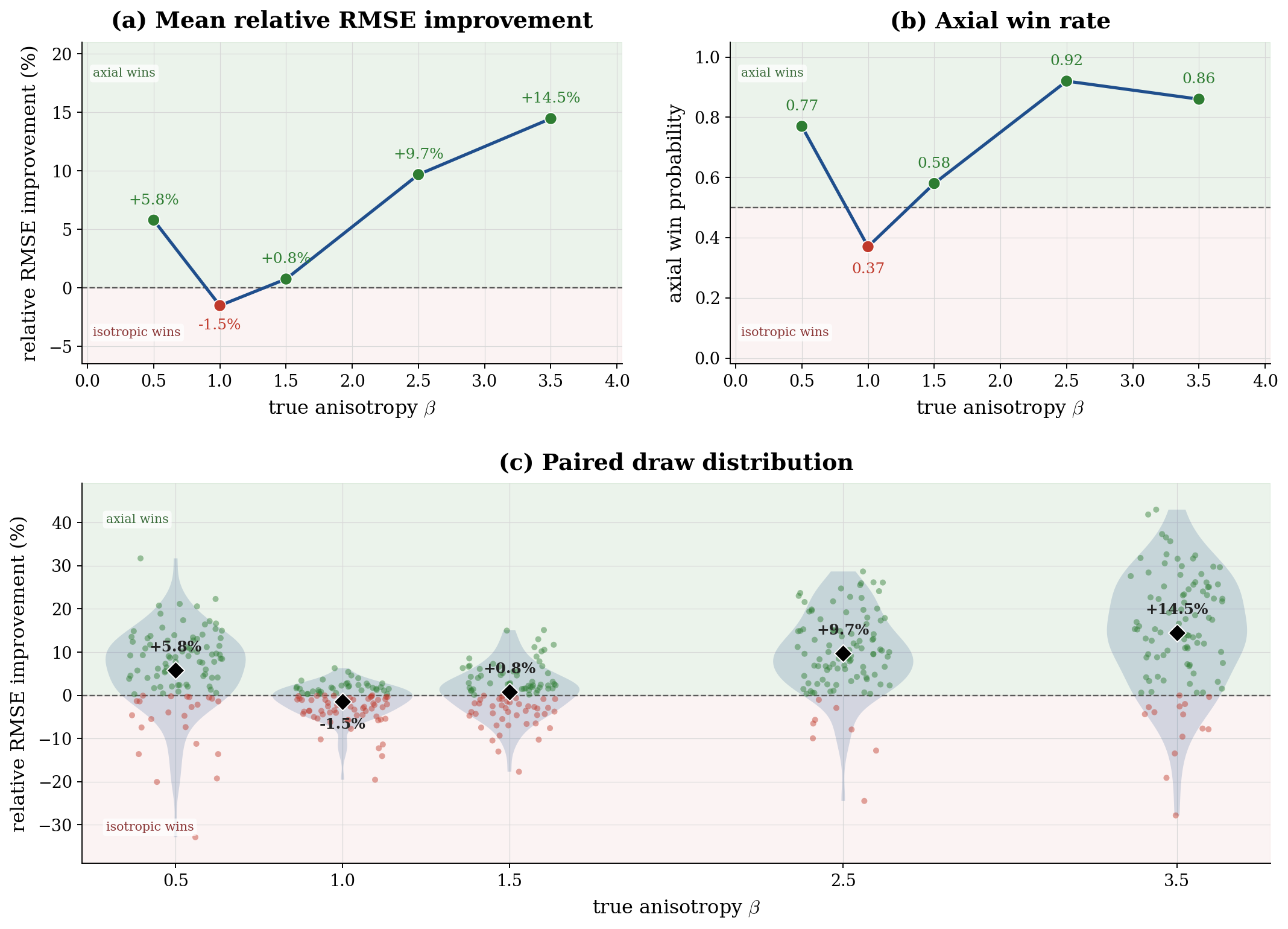}
\caption{Prediction under the proposed axial Mat\'ern deformation:
mean paired RMSE gain, win rate, and gain distributions across
deformation strengths.}
\label{fig:axial-strength}
\end{figure}

{
At \(\beta_0=1\), the generating covariance is isotropic and the
mean gain is \(-1.5\%\). The additional axial parameters do not
help on average in this case. For every tested non-isotropic
value of \(\beta_0\), the mean gain is positive. It is small at
\(\beta_0=1.5\), but reaches \(9.7\%\) and \(14.5\%\) at
\(\beta_0=2.5\) and \(3.5\), respectively. The axial model
performs worse than the isotropic model in some realizations.
On average, however, we observe larger gains as the deformation
becomes stronger for values \(\beta_0>1\).
}

%% ------------------------------------------------------------
%% Subsection: Prediction at matched mean correlation
%% ------------------------------------------------------------
\subsection{Prediction at matched mean correlation}
\label{subsec:matched-correlation}

{
Changing \(\beta\) changes both the shape of the covariance and
the mean correlation. In this experiment, we keep the mean
correlation fixed by adjusting the range.

For independent uniformly distributed points
\(S_1,S_2\in\mathbb S^2\), we define the mean pairwise correlation by
\[
M(\beta,\rho)
=
\mathbb E\!\left[
C_M\!\left(
\lVert T_{e_3,\beta}(S_1)-T_{e_3,\beta}(S_2)\rVert;\rho
\right)
\right].
\]
For each \(\beta_0\), we choose the generating range
\(\rho_{\beta_0}\) to match the mean correlation of the isotropic
model with range \(0.4\):
\[
M(\beta_0,\rho_{\beta_0})
=
M(1,0.4)
\approx0.0795.
\]
We still estimate the range freely in both fitted models.
}

\begin{figure}[!htbp]
\centering
\includegraphics[height=0.35\textheight]{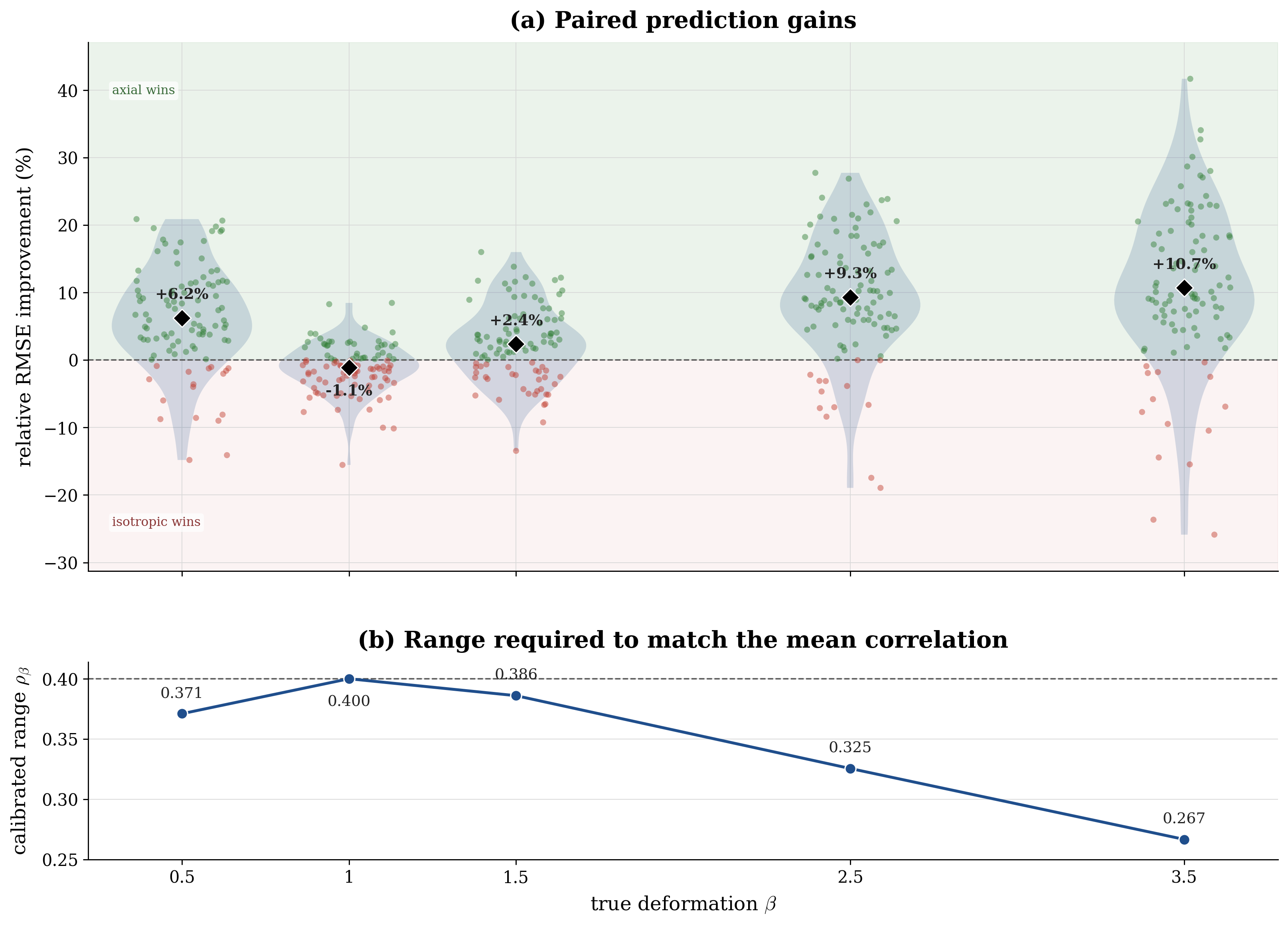}
\caption{Prediction after matching the mean pairwise correlation.
Panel~(a) shows the paired RMSE gain distributions for each
\(\beta_0\), with diamonds marking the means; panel~(b) shows the
corresponding calibrated generating ranges.}
\label{fig:matched-correlation}
\end{figure}

{
Figure~\ref{fig:matched-correlation} shows a mean gain of
\(-1.1\%\) at \(\beta_0=1\). For \(\beta_0=0.5\), \(2.5\),
and \(3.5\), the mean gains are \(6.2\%\), \(9.3\%\), and
\(10.7\%\), respectively. As in the previous experiment, the
axial model improves prediction on average for larger values of \(\beta_0\), even though the mean pairwise correlation is now held fixed.
}

%% ------------------------------------------------------------
%% Subsection: Cross-family comparison
%% ------------------------------------------------------------
\subsection{Cross-family comparison}
\label{subsec:cross-family}

{
We next test whether the prediction gain depends on the radial
family. We generate data with \(\beta_0=2.5\), using Mat\'ern,
Gaussian, or Wendland correlations, and fit all three radial
families in each case. For each combination in
Figure~\ref{fig:cross-family}, we compare an isotropic and an
axial fit from the same radial family.
}

\begin{figure}[!htbp]
\centering
\includegraphics[width=\textwidth]{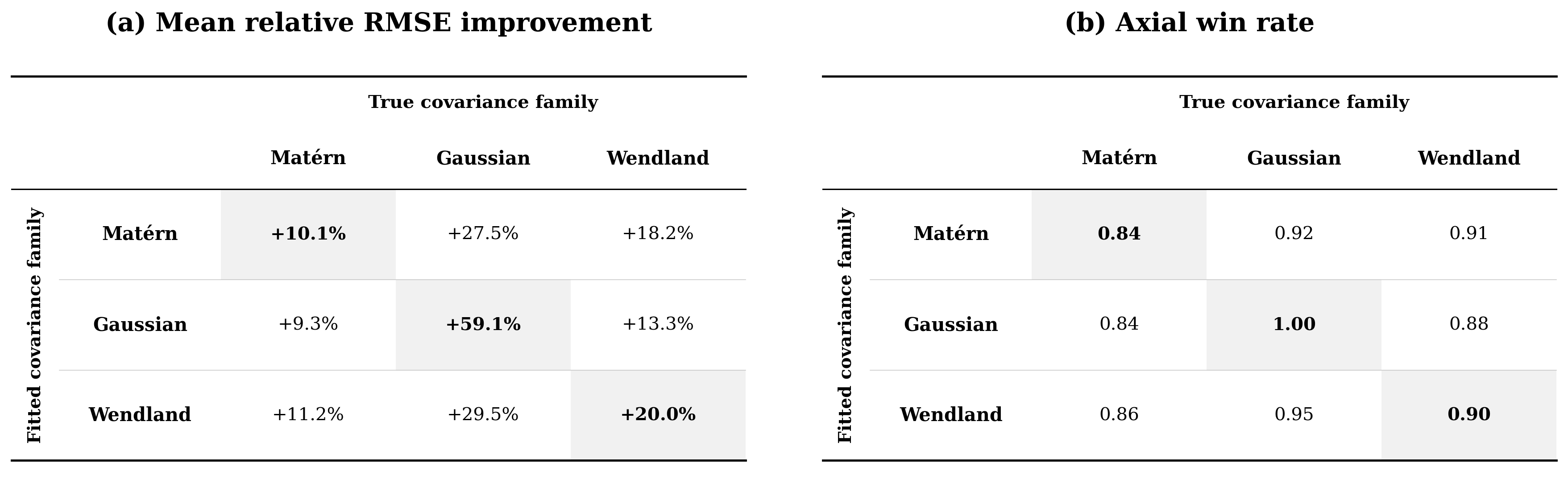}
\caption{Cross-family comparison under axial truth with \(\beta_0=2.5\).
The fitted radial families are listed on the left, and the generating
families across the top. Panel~(a) shows the mean relative RMSE
improvement, and panel~(b) the corresponding axial win rate.}
\label{fig:cross-family}
\end{figure}

{
All nine comparisons give a positive mean gain. The mean gains
range from \(9.3\%\) to \(59.1\%\), and the win rates range from
\(0.84\) to \(1.00\). The largest relative gains occur when the
data are generated with a Gaussian radial family, especially
when this family is also used for fitting. One possible
explanation is that these very smooth fields can be predicted
more accurately when the fitted model also captures the deformed
geometry.

The axial model also improves prediction on average when we
fit a different radial family from the one used to generate
the data.
}

%% ------------------------------------------------------------
%% Subsection: Robustness to alternative axial geometries
%% ------------------------------------------------------------
\subsection{Robustness to alternative axial geometries}
\label{subsec:alternative-warps}

{
We compare \(T_\beta\) with a midlatitude deformation
\(G_{\mathrm{mid}}\) and a flexible deformation \(G_{\mathrm{flex}}\).
Both take the form
\[
G\bigl(\omega(\theta,\lambda)\bigr)
=
\omega\bigl(g(\theta),\lambda\bigr),
\]
with
\[
\begin{aligned}
g_{\mathrm{mid}}(\theta;a)
&=
\theta-a\sin(2\theta)\bigl(1+\sin^2(2\theta)\bigr),
\\
g_{\mathrm{flex}}(\theta;c)
&=
\pi\,
\frac{\displaystyle\int_0^\theta e^{q_c(r)}\,\mathrm{d}r}
     {\displaystyle\int_0^\pi e^{q_c(r)}\,\mathrm{d}r},
\end{aligned}
\]
where \(q_c(\theta)=\sum_{j=1}^{3}c_j\cos(2j\theta)\).
The flexible family allows alternating bands of stretching
and compression.

For the parameter ranges used here, both maps are smooth and
strictly increasing, and satisfy
\[
g(0)=0,
\qquad
g(\pi)=\pi,
\qquad
g(\pi-\theta)=\pi-g(\theta).
\]
{
This means that the poles remain fixed and both
maps are symmetric about the equator. Both deformations also leave the longitude \(\lambda\)
unchanged, so the resulting covariance kernels are axially symmetric.

We use a Mat\'ern kernel and take \(T_\beta\)
with \(\beta=2\) as the reference generating deformation. The other two generating maps are
calibrated to the same root-mean-square angular displacement
over uniformly distributed points on the sphere. We also choose
the range of each generating model to give the mean pairwise
correlation \(M(1,0.4)\), as in
Subsection~\ref{subsec:matched-correlation}.
We fit all three axial families to data from each generating
geometry.
}

\begin{figure}[!htbp]
\centering
\includegraphics[width=0.72\textwidth]
{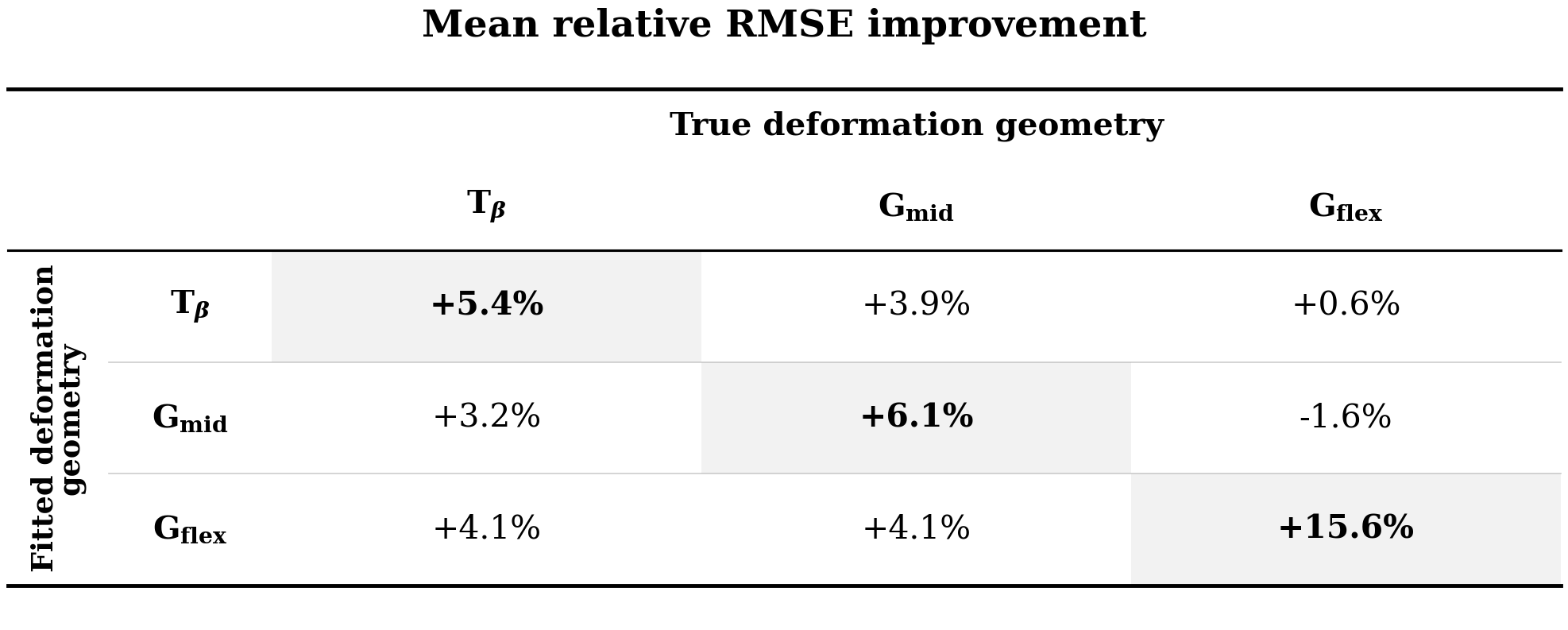}
\caption{Comparison of three axial deformation families. Each row corresponds
to a fitted deformation family, and each column to the family used
to generate the data. Entries show the mean relative RMSE improvement
over an isotropic model fitted to the same data.}
\label{fig:cross-deformation}
\end{figure}

{
In Figure~\ref{fig:cross-deformation}, the fit using the generating
deformation family gives the largest mean gain in each case.
Our \(T_\beta\) model also improves prediction on average for data
generated with \(G_{\mathrm{mid}}\), with a mean gain of \(3.9\%\).
It can therefore be useful even when the generating deformation
differs from the fitted family.

The banded example shows a limit of our model. Its mean
gain is only \(0.6\%\), compared with \(15.6\%\) for the flexible
fit. \(T\) is therefore not the best choice for every axial
covariance structure. However, the flexible family does not
perform best in every case either.

For illustration, Figure~\ref{fig:mid-spatial} shows a field
generated with $G_{\mathrm{mid}}$. To generate this field, we first
simulate an isotropic field using a spherical-harmonic expansion
truncated at degree $70$. We then evaluate the same realization
at the sample points and the points of a $90\times180$ grid
after applying $G_{\mathrm{mid}}$ to both sets of points.
Our $T_\beta$ model does not improve prediction everywhere.
Overall, however, its area-weighted grid RMSE is lower than
that of the isotropic model.
}

\begin{figure}[!htbp]
\centering
\includegraphics[width=\textwidth]
{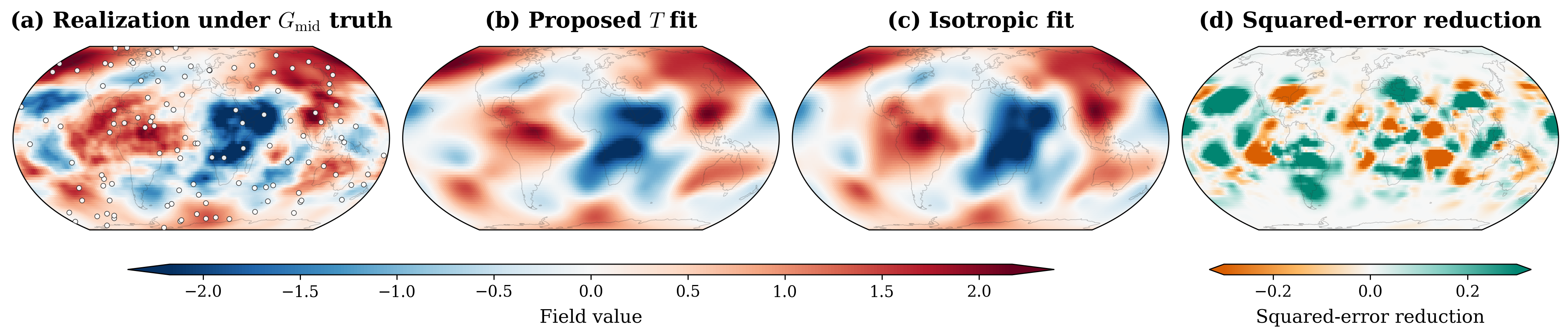}
\caption{From a separate set of ten realizations,
we select the one whose relative gain in area-weighted grid RMSE
is closest to the average gain. White circles mark the training
locations. Panel~(d) shows the difference in squared errors;
teal favours the proposed model and orange the isotropic model.}
\label{fig:mid-spatial}
\end{figure}
} %%============================================================
%% Section: Conclusion and outlook
%% ============================================================

\section{Conclusion and outlook}
\label{sec:summary}

{
Our construction uses a linear scaling along an axis, following
the idea in \cite{BeatsonDavydovLevesley2010}, but then normalizes
the result back to the sphere. Applying an isotropic kernel to
the transformed points gives an axially symmetric kernel.
We express the spherical-harmonic coefficients of this kernel
in terms of the original isotropic coefficients and give a finite
formula for the coupling coefficients. For each fixed input mode,
we prove that these coupling coefficients decay geometrically.
Furthermore, the Sobolev order of the native space and the
interpolation convergence order remain the same as in the
isotropic case under suitable spectral assumptions.
On average, our model predicts better than the isotropic model
in several of the tested axial cases. We also see this when
we fix the mean correlation or fit a different radial family
from the one used to generate the data. However, the banded
example shows that a single global deformation does not give
the best predictions for every axial covariance structure.
Our model is an intermediate choice between isotropy and more
flexible covariance families. There are several future directions to follow from our discussion. 
Future work could focus on estimation with larger datasets
and applications to global data, to improve predictions where currently isotropic kernels are used. We could also study how
accurately the axis and deformation strength can be estimated
from the observations.

A second important continuation would be to employ our results to derive explicit bounds on the $L_2$-error resulting from using a finite series expansion of the kernel for simulation, in the spirit of \cite{alegria2020karhunen}. 

Another question is how to build
deformations with more parameters while keeping an explicit
harmonic representation. 

We note that the results can also be viewed in a second way by focusing on the transformed point set introduced in section \ref{subsec:interpolation}. This way, the transformation can be used to distribute the points more favorably on the sphere as a variant of the fake nodes approach.
}

%% ============================================================
%% Section: Acknowledgements
%% ============================================================
\section*{Acknowledgements}

The geometric viewpoint developed in this paper was inspired in part by the
work of the late Professor Rick Beatson. In particular, the study by
Beatson, Davydov, and Levesley on anisotropically transformed radial basis
functions \cite{BeatsonDavydovLevesley2010} motivated the idea of generating
anisotropy by transforming the underlying geometry. Although their setting
differs from the normalized spherical deformation studied here, this general
perspective shaped the approach taken in this paper. We also
acknowledge Professor Beatson's lasting contributions to approximation
theory, radial basis functions, and kernel methods.

%% ============================================================
%% Section: Declarations
%% ============================================================
\section*{Declarations}

%% ------------------------------------------------------------
%% Subsection: Funding
%% ------------------------------------------------------------
\subsection*{Funding}
The work of Fabian Stallbauer and Janin J{\"a}ger has been supported
by the Deutsche Forschungsgemeinschaft (DFG, German Research Foundation)
under project number 461449252,
\emph{Kernel interpolation on Riemannian manifolds}.

%% ------------------------------------------------------------
%% Subsection: Author contributions
%% ------------------------------------------------------------
\subsection*{Author contributions}
Janin J{\"a}ger, as supervisor, contributed to the development of the
main ideas, proposed theorems and experiments, and contributed to the
structure and organization of the paper. Fabian Stallbauer developed
the proofs and conducted the numerical experiments. Both authors
discussed the results and contributed to the revision and finalization
of the manuscript.

%% ------------------------------------------------------------
%% Subsection: Competing interests
%% ------------------------------------------------------------
\subsection*{Competing interests}
The authors have no competing interests to declare.

%% ------------------------------------------------------------
%% Subsection: Data and code availability
%% ------------------------------------------------------------
\subsection*{Data and code availability}
All data used in the numerical experiments are simulated.
The code, numerical results, and scripts for reproducing the figures
are available at
\url{https://github.com/fabianstallbauer/axial-spherical-kernels}.

%% ============================================================
%% Bibliography
%% ============================================================
\bibliographystyle{plain}
\bibliography{references_ordered}

%% ============================================================
%% Appendices
%% ============================================================
\appendix

%%===========================================================
%% Section: Derivation of the explicit connection-coefficient formula
%% ============================================================
\section{Derivation of the explicit connection-coefficient formula}
\label{app:explicit-S2}

This appendix proves Theorem~\ref{thm:explicit-S2}. The proof proceeds in three steps: the finite associated-Legendre expansion, the parity rule, and the evaluation of the remaining scalar integrals by Euler's representation \eqref{eq:Euler-hypergeometric-integral}.

\begin{proof}[Proof of Theorem~\ref{thm:explicit-S2}]
Let \(a\ge\mu\ge0\) be integers. Substituting
\[
 q=a-\mu,
 \qquad
 \nu=\mu+\frac12
\]
into the finite Gegenbauer expansion
\eqref{eq:Gegenbauer-finite-expansion}, and then using
\eqref{eq:associated-Legendre-Gegenbauer}, gives
\[
 \begin{aligned}
 P_a^\mu(x)
 &=
 (-1)^\mu\frac{(2\mu)!}{2^\mu\mu!}
 (1-x^2)^{\mu/2}
 \\
 &\quad\times
 \sum_{k=0}^{\lfloor(a-\mu)/2\rfloor}
 (-1)^k
 \frac{\left(\mu+\frac12\right)_{a-\mu-k}}
 {k!(a-\mu-2k)!}
 (2x)^{a-\mu-2k}.
 \end{aligned}
\]
The identity
\[
 \left(\mu+\frac12\right)_{a-\mu-k}
 =
 \frac{(2a-2k)!\,\mu!}
 {4^{a-\mu-k}(a-k)!(2\mu)!}
\]
reduces the coefficient of \(x^{a-\mu-2k}\) to
\[
 \frac{(-1)^{\mu+k}(2a-2k)!}
 {2^a k!(a-k)!(a-\mu-2k)!}.
\]
Using the notation introduced in \eqref{eq:gamma-def-main}, we obtain
\begin{equation}
\label{eq:app-associated-Legendre-finite}
 P_a^\mu(x)
 =
 (1-x^2)^{\mu/2}
 \sum_{k=0}^{R_{a,\mu}}
 \gamma_{a,\mu,k}x^{a-\mu-2k}.
\end{equation}

Set \(\mu=|m|\). By Theorem~\ref{thm:no-mixing-S2},
\begin{equation}
\label{eq:app-B-one-dimensional}
 B_n^{(\ell,m)}(\beta)
 =
 2\pi\mathcal N_{\ell\mu}\mathcal N_{n\mu}
 \int_{-1}^{1}
 P_\ell^\mu\bigl(\widetilde x_\beta(x)\bigr)
 P_n^\mu(x)\,\dd x.
\end{equation}
The map \(\widetilde x_\beta\) is odd, and
\eqref{eq:app-associated-Legendre-finite} gives
\[
 P_a^\mu(-x)
 =
 (-1)^{a-\mu}P_a^\mu(x)
 =
 (-1)^{a+\mu}P_a^\mu(x).
\]
Consequently,
\[
 P_\ell^\mu\bigl(\widetilde x_\beta(-x)\bigr)P_n^\mu(-x)
 =
 (-1)^{\ell+n}
 P_\ell^\mu\bigl(\widetilde x_\beta(x)\bigr)P_n^\mu(x).
\]
The integrand in \eqref{eq:app-B-one-dimensional} is therefore odd when
\(\ell+n\) is odd. Hence
\begin{equation}
\label{eq:app-parity-rule}
 \ell+n\ \text{odd}
 \qquad\Longrightarrow\qquad
 B_n^{(\ell,m)}(\beta)=0.
\end{equation}

Suppose now that \(\ell+n\) is even. From the definition of
\(\widetilde x_\beta\),
\begin{equation}
\label{eq:app-pullback-identities}
 1-\widetilde x_\beta(x)^2
 =
 \frac{1-x^2}{1+(\beta^2-1)x^2},
 \qquad
 \widetilde x_\beta(x)^q
 =
 \beta^q x^q
 \bigl(1+(\beta^2-1)x^2\bigr)^{-q/2},
 \quad q\in\mathbb N_0.
\end{equation}
Applying \eqref{eq:app-associated-Legendre-finite} with \(a=\ell\) and
\(a=n\), respectively, yields
\begin{align*}
 P_\ell^\mu\bigl(\widetilde x_\beta(x)\bigr)
 &=
 (1-x^2)^{\mu/2}
 \sum_{k=0}^{R_{\ell,\mu}}
 \gamma_{\ell,\mu,k}
 \beta^{\ell-\mu-2k}
 x^{\ell-\mu-2k}
 \bigl(1+(\beta^2-1)x^2\bigr)^{-(\ell-2k)/2},
 \\
 P_n^\mu(x)
 &=
 (1-x^2)^{\mu/2}
 \sum_{j=0}^{R_{n,\mu}}
 \gamma_{n,\mu,j}x^{n-\mu-2j}.
\end{align*}
Substituting these expressions into
\eqref{eq:app-B-one-dimensional} gives
\begin{equation}
\label{eq:app-B-with-J}
 B_n^{(\ell,m)}(\beta)
 =
 2\pi\mathcal N_{\ell\mu}\mathcal N_{n\mu}
 \sum_{k=0}^{R_{\ell,\mu}}
 \sum_{j=0}^{R_{n,\mu}}
 \gamma_{\ell,\mu,k}\gamma_{n,\mu,j}
 \beta^{\ell-\mu-2k}
 J_{p_{k,j},s_k}^{(\mu)}(\beta),
\end{equation}
where
\[
 s_k=\frac{\ell-2k}{2},
 \qquad
 p_{k,j}
 =
 \frac{\ell+n-2\mu-2(k+j)}{2},
\]
and
\[
 J_{p,s}^{(\mu)}(\beta)
 :=
 \int_{-1}^{1}
 x^{2p}(1-x^2)^\mu
 \bigl(1+(\beta^2-1)x^2\bigr)^{-s}
 \,\dd x.
\]
The parity assumption gives \(p_{k,j}\in\mathbb Z\), while the summation
bounds imply
\[
 2p_{k,j}
 =
 \ell+n-2\mu-2(k+j)\ge0.
\]
Thus \(p_{k,j}\in\mathbb N_0\).

The integrand defining \(J_{p,s}^{(\mu)}(\beta)\) is even. With \(t=x^2\),
\[
 \begin{aligned}
 J_{p,s}^{(\mu)}(\beta)
 &=
 2\int_0^1
 x^{2p}(1-x^2)^\mu
 \bigl(1+(\beta^2-1)x^2\bigr)^{-s}
 \,\dd x
 \\
 &=
 \int_0^1
 t^{p-\frac12}(1-t)^\mu
 \bigl(1+(\beta^2-1)t\bigr)^{-s}
 \,\dd t.
 \end{aligned}
\]
To apply \eqref{eq:Euler-hypergeometric-integral}, set
\[
 a=s,
 \qquad
 b=p+\frac12,
 \qquad
 c=p+\mu+\frac32,
 \qquad
 z=1-\beta^2.
\]
Here
\[
 b>0,
 \qquad
 c-b=\mu+1>0,
 \qquad
 z<1.
\]
Moreover,
\[
 1-zt
 =
 1+(\beta^2-1)t>0,
 \qquad 0\le t\le1.
\]
Euler's formula therefore gives
\begin{equation}
\label{eq:app-J-hypergeometric}
 J_{p,s}^{(\mu)}(\beta)
 =
 \Bfun\left(p+\frac12,\mu+1\right)
 {}_2F_1\left(
 s,p+\frac12;
 p+\mu+\frac32;
 1-\beta^2
 \right).
\end{equation}
Substituting \eqref{eq:app-J-hypergeometric} into
\eqref{eq:app-B-with-J} gives \eqref{eq:explicit-B-S2}. Together with
\eqref{eq:app-parity-rule}, this proves the theorem.
\end{proof}

%% ============================================================
%% Section: Proof of geometric fixed-mode decay
%% ============================================================

\section{Proof of geometric fixed-mode decay}
\label{app:decay-S2}

This appendix proves Theorem~\ref{thm:fixed-mode-decay}. It combines Wang's coefficient estimate, stated in the Gegenbauer convention used here, with an associated-Legendre integral estimate and with the Bernstein ellipses on which the transformed profile \(H_{\ell,\mu,\beta}\) is holomorphic.

Throughout this appendix, the input mode \((\ell,m)\) is fixed and
\(\mu=|m|\). Constants may depend on \(\ell\), \(m\), \(\beta\), and the
chosen Bernstein ellipse, but not on the output degree \(n\).

%% ------------------------------------------------------------
%% Subsection: Gegenbauer coefficients
%% ------------------------------------------------------------
\subsection*{Gegenbauer coefficients}

For \(\nu>0\), define the Gegenbauer weight by
\begin{equation}
\label{eq:Gegenbauer-weight}
 w_\nu(x)
 =
 (1-x^2)^{\nu-\frac12}.
\end{equation}
The Gegenbauer polynomials satisfy
\begin{equation}
\label{eq:Gegenbauer-orthogonality}
 \int_{-1}^{1}
 C_q^{(\nu)}(x)
 C_r^{(\nu)}(x)
 w_\nu(x)\,\dd x
 =
 h_q^{(\nu)}\delta_{qr},
\end{equation}
where
\begin{equation}
\label{eq:Gegenbauer-norm}
 h_q^{(\nu)}
 =
 \frac{
 2^{1-2\nu}\pi\Gamma(q+2\nu)
 }{
 \Gamma(\nu)^2
 \Gamma(q+1)
 (q+\nu)
 }.
\end{equation}

For \(g\in L^2((-1,1),w_\nu(x)\,\dd x)\), define its \(q\)-th
Gegenbauer coefficient by
\begin{equation}
\label{eq:Gegenbauer-coefficient}
 \mathfrak a_q^{(\nu)}(g)
 :=
 \frac{1}{h_q^{(\nu)}}
 \int_{-1}^{1}
 g(x)C_q^{(\nu)}(x)w_\nu(x)\,\dd x,
 \qquad q\in\mathbb N_0.
\end{equation}

%% ------------------------------------------------------------
%% Subsection: Wang's estimate and its application
%% ------------------------------------------------------------
\subsection*{Wang's estimate and its application}

For \(\rho>1\), let \(\partial\mathcal E_\rho\) be the Bernstein ellipse
\begin{equation}
\label{eq:app-Bernstein-ellipse}
 \partial\mathcal E_\rho
 =
 \left\{
 \frac12\left(
 \rho e^{i\theta}+\rho^{-1}e^{-i\theta}
 \right):
 0\le\theta\le2\pi
 \right\},
\end{equation}
and let \(\mathcal E_\rho\) denote the compact region enclosed by this
curve. Equivalently, its boundary is parametrized by
\[
 z=\frac12(w+w^{-1}),
 \qquad |w|=\rho.
\]

For a function \(g\) holomorphic in a neighbourhood of
\(\mathcal E_\rho\), set
\[
 M_\rho(g)
 :=
 \max_{z\in\partial\mathcal E_\rho}|g(z)|.
\]

For \(q\ge1\), define
\begin{equation}
\label{eq:app-Wang-Upsilon}
 \Upsilon_q^{1,\nu}
 :=
 \exp\left(
 \frac{1-\nu}{2(q+\nu-1)}
 +
 \frac{1}{12q}
 \right),
\end{equation}
and
\begin{equation}
\label{eq:app-Wang-Xi}
 \Xi_{q,\rho,\nu}(g)
 :=
 \frac{2\Gamma(\nu)M_\rho(g)\Upsilon_q^{1,\nu}}{\pi}
 \left[
 (\rho+\rho^{-1})
 +
 \left(\frac{\pi}{2}-1\right)(\rho-\rho^{-1})
 \right].
\end{equation}

\begin{theorem}[Wang's estimate]
\label{thm:app-Wang-Gegenbauer}
Let \(\nu>0\), \(\rho>1\), and let \(g\) be holomorphic in a
neighbourhood of \(\mathcal E_\rho\). Then, for every \(q\ge1\),
\begin{equation}
\label{eq:app-Wang-full-estimate}
 \bigl|\mathfrak a_q^{(\nu)}(g)\bigr|
 \le
 \begin{cases}
 \Xi_{q,\rho,\nu}(g)
 (1-\rho^{-2})^{\nu-1}
 q^{1-\nu}\rho^{-(q+1)},
 & 0<\nu\le1,\\[0.8em]
 \Xi_{q,\rho,\nu}(g)
 (1+\rho^{-2})^{\nu-1}
 q^{1-\nu}\rho^{-(q+1)},
 & \nu>1.
 \end{cases}
\end{equation}
\end{theorem}

This is Wang's estimate
\cite[Theorem~4.3]{Wang2016Gegenbauer},
expressed in the coefficient convention fixed above. The factor
\(\Upsilon_q^{1,\nu}\) is obtained from
Wang~\cite[Lemma~4.2]{Wang2016Gegenbauer}
with \(a=1\) and \(b=\nu\).

In the application below, we fix the input mode \((\ell,m)\)
and the deformation parameter \(\beta\), so \(g\) and \(\nu\)
are fixed. We choose an admissible Bernstein ellipse and
simplify the estimate for these fixed choices.

For \(q\ge1\), we have
\[
 q+\nu-1\ge\nu,
 \qquad
 q^{-1}\le1.
\]
Therefore,
\[
 \Upsilon_q^{1,\nu}
 \le
 \exp\left(
 \frac{|1-\nu|}{2\nu}
 +
 \frac{1}{12}
 \right).
\]
Thus \(\Xi_{q,\rho,\nu}(g)\) is bounded independently of \(q\).
The factors involving \(1\pm\rho^{-2}\) are also fixed, and the
additional factor \(\rho^{-1}\) can be absorbed into the constant.
Therefore,
\begin{equation}
\label{eq:app-Wang-reduced}
 \bigl|\mathfrak a_q^{(\nu)}(g)\bigr|
 \le
 C_{g,\nu,\rho}
 q^{1-\nu}\rho^{-q},
 \qquad q\ge1,
\end{equation}
where \(C_{g,\nu,\rho}>0\) is independent of \(q\).

\begin{corollary}
\label{cor:app-associated-Legendre-projection}
Fix \(\mu\in\mathbb N_0\) and \(\rho>1\). If \(g\) is holomorphic in a
neighbourhood of \(\mathcal E_\rho\), then there exists
\(C_{g,\mu,\rho}>0\), independent of \(n\), such that
\begin{equation}
\label{eq:app-associated-Legendre-projection}
 \mathcal N_{n\mu}
 \left|
 \int_{-1}^{1}
 g(x)P_n^\mu(x)(1-x^2)^{\mu/2}\,\dd x
 \right|
 \le
 C_{g,\mu,\rho}\rho^{-n},
 \qquad n\ge\mu.
\end{equation}
\end{corollary}

\begin{proof}
Set
\[
 \nu=\mu+\frac12,
 \qquad
 q=n-\mu.
\]
By \eqref{eq:associated-Legendre-Gegenbauer},
\[
 P_n^\mu(x)
 =
 d_\mu(1-x^2)^{\mu/2}C_q^{(\nu)}(x).
\]
Since
\[
 \nu-\frac12=\mu,
\]
the coefficient definition
\eqref{eq:Gegenbauer-coefficient} gives
\begin{equation}
\label{eq:app-Legendre-projection-as-Gegenbauer}
 \int_{-1}^{1}
 g(x)P_n^\mu(x)(1-x^2)^{\mu/2}\,\dd x
 =
 d_\mu h_q^{(\nu)}
 \mathfrak a_q^{(\nu)}(g).
\end{equation}

The case \(q=0\), equivalently \(n=\mu\), is covered by enlarging the
constant. We may therefore assume \(q\ge1\).

For fixed \(\mu\),
\[
 \frac{(n-\mu)!}{(n+\mu)!}
 =
 \left(
 \prod_{r=-\mu+1}^{\mu}(n+r)
 \right)^{-1}
 =
 O(n^{-2\mu}),
\]
where an empty product is interpreted as \(1\). Hence
\begin{equation}
\label{eq:app-normalization-growth}
 \mathcal N_{n\mu}
 =
 \left(
 \frac{2n+1}{4\pi}
 \frac{(n-\mu)!}{(n+\mu)!}
 \right)^{1/2}
 =
 O\bigl(n^{\frac12-\mu}\bigr).
\end{equation}

Using \(\nu=\mu+\frac12\) in
\eqref{eq:Gegenbauer-norm}, we obtain
\[
 h_q^{(\nu)}
 =
 \frac{2^{-2\mu}\pi}
 {\Gamma(\mu+\frac12)^2}
 \frac{\Gamma(q+2\mu+1)}
 {\Gamma(q+1)(q+\mu+\frac12)}.
\]
Since
\[
 \frac{\Gamma(q+2\mu+1)}{\Gamma(q+1)}
 =
 \prod_{r=1}^{2\mu}(q+r)
 =
 O(q^{2\mu}),
\]
it follows that
\begin{equation}
\label{eq:app-Gegenbauer-norm-growth}
 h_q^{(\nu)}
 =
 O(q^{2\mu-1}).
\end{equation}

Because
\[
 1-\nu=\frac12-\mu,
\]
the algebraic factor obtained from
\eqref{eq:app-Wang-reduced} satisfies
\[
 \mathcal N_{n\mu}
 h_q^{(\nu)}
 q^{1-\nu}
 =
 O\left(
 n^{\frac12-\mu}
 q^{\mu-\frac12}
 \right).
\]
If \(\mu=0\), then \(q=n\). If \(\mu\ge1\), then
\(0<q/n\le1\) and \(\mu-\frac12>0\). Thus
\begin{equation}
\label{eq:app-algebraic-factor-bounded}
 \mathcal N_{n\mu}
 h_q^{(\nu)}
 q^{1-\nu}
 =
 O(1).
\end{equation}

Because \(d_\mu\) is fixed, combining
\eqref{eq:app-Legendre-projection-as-Gegenbauer},
\eqref{eq:app-Wang-reduced}, and
\eqref{eq:app-algebraic-factor-bounded} gives
\[
 \mathcal N_{n\mu}
 \left|
 \int_{-1}^{1}
 g(x)P_n^\mu(x)(1-x^2)^{\mu/2}\,\dd x
 \right|
 \le
 C_{g,\mu,\rho}\rho^{-q}.
\]
Finally,
\[
 \rho^{-q}
 =
 \rho^\mu\rho^{-n},
\]
and the fixed factor \(\rho^\mu\) is absorbed into the constant.
\end{proof}

%% ------------------------------------------------------------
%% Subsection: Analyticity of the transformed profile
%% ------------------------------------------------------------
\subsection*{Analyticity of the transformed profile}

We now identify the Bernstein ellipses on which the transformed profile is
holomorphic. By
\eqref{eq:H-factor-main}--\eqref{eq:H-def-main},
\[
 P_\ell^\mu\bigl(\widetilde x_\beta(x)\bigr)
 =
 (1-x^2)^{\mu/2}H_{\ell,\mu,\beta}(x),
\]
where
\[
 H_{\ell,\mu,\beta}(x)
 =
 \sum_{k=0}^{R_{\ell,\mu}}
 \gamma_{\ell,\mu,k}
 \beta^{\ell-\mu-2k}
 x^{\ell-\mu-2k}
 \bigl(1+(\beta^2-1)x^2\bigr)^{-(\ell-2k)/2}.
\]
\begin{lemma}
\label{lem:app-H-analyticity}
Let \(\beta>0\), \(\beta\neq1\), and set
\[
 \rho_\beta
 :=
 \sqrt{\frac{1+\beta}{|1-\beta|}}.
\]
Then \(H_{\ell,\mu,\beta}\) extends holomorphically to a neighbourhood of
\(\mathcal E_\rho\) for every \(1<\rho<\rho_\beta\).
\end{lemma}

\begin{proof}
Set
\[
 q_\beta(z):=1+(\beta^2-1)z^2.
\]
Possible singularities in the finite representation above occur only at the
zeros of \(q_\beta\).

The zeros satisfy
\[
 z^2=(1-\beta^2)^{-1}.
\]
If \(0<\beta<1\), the positive zero is
\(z_0=(1-\beta^2)^{-1/2}\). Its exterior preimage under
\(z=\frac12(w+w^{-1})\) is
\[
 w_0
 =
 z_0+\sqrt{z_0^2-1}
 =
 \sqrt{\frac{1+\beta}{1-\beta}}.
\]
If \(\beta>1\), one zero is
\(z_0=i(\beta^2-1)^{-1/2}\). The corresponding exterior preimage satisfies
\[
 |w_0|
 =
 \frac{1}{\sqrt{\beta^2-1}}
 +
 \sqrt{1+\frac{1}{\beta^2-1}}
 =
 \sqrt{\frac{1+\beta}{\beta-1}}.
\]
Thus the zeros of \(q_\beta\) lie on the Bernstein ellipse with parameter
\begin{equation}
\label{eq:rho-beta-from-ellipse}
 \rho_\beta
 =
 \sqrt{\frac{1+\beta}{|1-\beta|}}.
\end{equation}

Fix \(1<\rho<\rho_\beta\) and choose
\(\rho<\rho_*<\rho_\beta\). The function \(q_\beta\) is holomorphic and
nonvanishing on the simply connected domain
\(\operatorname{int}\mathcal E_{\rho_*}\). Hence it has a holomorphic
square root; see \cite[Ch.~13]{Rudin1987}. We choose \(s_\beta\) such that
\[
 s_\beta(z)^2=q_\beta(z),
 \qquad
 s_\beta(0)=1.
\]
For \(x\in[-1,1]\), we have \(q_\beta(x)>0\). Continuity and the
normalization at \(0\) therefore imply
\[
 s_\beta(x)=\sqrt{q_\beta(x)}.
\]
We define
\[
 q_\beta(z)^{-(\ell-2k)/2}
 :=
 s_\beta(z)^{-(\ell-2k)}.
\]
With this definition, the finite formula above gives a holomorphic extension
of \(H_{\ell,\mu,\beta}\) to
\(\operatorname{int}\mathcal E_{\rho_*}\). Since
\(\mathcal E_\rho\subset\operatorname{int}\mathcal E_{\rho_*}\), this
extension is holomorphic in a neighbourhood of \(\mathcal E_\rho\).
\end{proof}

\begin{proof}[Proof of Theorem~\ref{thm:fixed-mode-decay}]
Let \(1<\rho<\rho_\beta\). By
Lemma~\ref{lem:app-H-analyticity},
\(H_{\ell,\mu,\beta}\) is holomorphic in a neighbourhood of
\(\mathcal E_\rho\). Substituting \eqref{eq:H-factor-main} into
\eqref{eq:B-integral-S2} gives
\[
 B_n^{(\ell,m)}(\beta)
 =
 2\pi\mathcal N_{\ell\mu}\mathcal N_{n\mu}
 \int_{-1}^{1}
 H_{\ell,\mu,\beta}(x)
 P_n^\mu(x)(1-x^2)^{\mu/2}\,\dd x.
\]
The factor \(2\pi\mathcal N_{\ell\mu}\) is independent of \(n\).
Applying Corollary~\ref{cor:app-associated-Legendre-projection} with
\(g=H_{\ell,\mu,\beta}\) gives
\[
 \begin{aligned}
 \bigl|B_n^{(\ell,m)}(\beta)\bigr|
 &\le
 2\pi\mathcal N_{\ell\mu}
 C_{H_{\ell,\mu,\beta},\mu,\rho}\rho^{-n}
 \\
 &=
 C_{\ell,m,\beta,\rho}\rho^{-n},
 \qquad n\ge\mu,
 \end{aligned}
\]
where \(C_{\ell,m,\beta,\rho}\) is independent of \(n\). This proves
\eqref{eq:fixed-mode-geometric-decay}.

Taking \(n\)-th roots gives
\[
 \bigl|B_n^{(\ell,m)}(\beta)\bigr|^{1/n}
 \le
 C_{\ell,m,\beta,\rho}^{1/n}\rho^{-1}.
\]
Since \(C_{\ell,m,\beta,\rho}^{1/n}\to1\),
\[
 \limsup_{n\to\infty}
 \bigl|B_n^{(\ell,m)}(\beta)\bigr|^{1/n}
 \le
 \rho^{-1}.
\]
Since \(\rho\in(1,\rho_\beta)\) was arbitrary, we may let
\(\rho\uparrow\rho_\beta\). This gives
\[
 \limsup_{n\to\infty}
 \bigl|B_n^{(\ell,m)}(\beta)\bigr|^{1/n}
 \le
 \rho_\beta^{-1}
 =
 \sqrt{\frac{|1-\beta|}{1+\beta}}.
\]
This proves \eqref{eq:limsup-decay}.
\end{proof}

%% ============================================================
%% Section: Proof of Sobolev stability
%% ============================================================
\section{Proof of Sobolev stability}
\label{app:native-transport-proof}

On compact manifolds, Sobolev stability under smooth diffeomorphisms follows
from the local definition of Sobolev spaces and the invariance of these spaces
under smooth coordinate changes; see \cite[Ch.~4]{Taylor2011}. We give a
short atlas argument on \(\Sph^2\) to connect this result with the
Fourier--Laplace Sobolev norm used in this paper. For the equivalent spectral
description on the sphere, see
\cite[Ch.~1]{HubbertLeGiaMorton2015}.

Let \(T:\Sph^2\to\Sph^2\) be a \(C^\infty\)-diffeomorphism, and define
\[
 C_Tf:=f\circ T.
\]

%% ------------------------------------------------------------
%% Subsection: Integer Sobolev orders
%% ------------------------------------------------------------
\subsection{Integer Sobolev orders}
\label{subsec:integer-sobolev-stability-appendix}

We first state the integer-order result in the notation used throughout the
paper.

\begin{theorem}
\label{thm:sobolev-stability-appendix}
For every integer \(k\ge0\), the composition operator \(C_T\) extends
uniquely to a bounded isomorphism on \(H^k(\Sph^2)\). Equivalently, there
exist constants \(A_{k,T},B_{k,T}>0\) such that
\[
 A_{k,T}\|f\|_{H^k(\Sph^2)}
 \le
 \|f\circ T\|_{H^k(\Sph^2)}
 \le
 B_{k,T}\|f\|_{H^k(\Sph^2)}
\]
for every \(f\in H^k(\Sph^2)\).
\end{theorem}

\begin{proof}
We first prove the estimate for \(f\in C^\infty(\Sph^2)\). Choose a finite
smooth atlas and a subordinate smooth partition of unity as in
Lee~\cite[Ch.~1--2]{Lee2003},
\[
 \mathcal A
 =
 \{(U_j,\kappa_j,\chi_j)\}_{j=1}^N,
 \qquad
 \kappa_j:U_j\to
 \Omega_j:=\kappa_j(U_j)\subset\mathbb R^2,
\]
with \(\operatorname{supp}\chi_j\subset U_j\). For
\(k\in\mathbb N_0\), define
\begin{equation}
\label{eq:atlas-norm-D}
 \|f\|_{H^k(\Sph^2),\mathcal A}^2
 :=
 \sum_{j=1}^N
 \left\|
 (\chi_jf)\circ\kappa_j^{-1}
 \right\|_{H^k(\Omega_j)}^2.
\end{equation}
This is the standard localized Sobolev norm on a compact manifold. Smooth
coordinate changes preserve Sobolev norms up to constant factors. Therefore,
different finite smooth atlases and subordinate partitions of unity define
equivalent norms; see \cite[Ch.~4]{Taylor2011}.

On \(\Sph^2\), the Sobolev norm can also be expressed through the
spherical-harmonic coefficients, as in \eqref{eq:sobolev-S2}; see
\cite[Ch.~1]{HubbertLeGiaMorton2015}. This spectral norm is equivalent to
the localized atlas norm. Hence, for the chosen atlas \(\mathcal A\),
there exist constants
\(c_{\mathcal A,k},C_{\mathcal A,k}>0\) such that
\begin{equation}
\label{eq:atlas-spectral-equivalence-D}
 c_{\mathcal A,k}\|f\|_{H^k(\Sph^2)}
 \le
 \|f\|_{H^k(\Sph^2),\mathcal A}
 \le
 C_{\mathcal A,k}\|f\|_{H^k(\Sph^2)}.
\end{equation}
Transport the atlas by \(T\):
\[
 V_j:=T^{-1}(U_j),
 \qquad
 \widetilde\kappa_j:=\kappa_j\circ T,
 \qquad
 \widetilde\chi_j:=\chi_j\circ T.
\]
Since \(T\) is a diffeomorphism,
\[
 \widetilde{\mathcal A}
 :=
 \{(V_j,\widetilde\kappa_j,\widetilde\chi_j)\}_{j=1}^N
\]
is again a finite smooth atlas with a subordinate smooth partition of
unity. Moreover,
\begin{equation}
\label{eq:transported-chart-inverse-D}
 \widetilde\kappa_j^{-1}
 =
 T^{-1}\circ\kappa_j^{-1},
 \qquad
 \widetilde\kappa_j(V_j)=\Omega_j.
\end{equation}
For every \(j\), the corresponding local representatives satisfy
\begin{align}
\label{eq:local-rep-identity-D}
 (\widetilde\chi_j(f\circ T))
 \circ
 \widetilde\kappa_j^{-1}
 &=
 \bigl((\chi_j\circ T)(f\circ T)\bigr)
 \circ T^{-1}\circ\kappa_j^{-1}
 \notag\\
 &=
 (\chi_jf)\circ\kappa_j^{-1}.
\end{align}
Hence the localized norms agree exactly:
\begin{equation}
\label{eq:exact-atlas-norm-identity-D}
 \|f\circ T\|_{H^k(\Sph^2),\widetilde{\mathcal A}}
 =
 \|f\|_{H^k(\Sph^2),\mathcal A}.
\end{equation}

We first apply \eqref{eq:atlas-spectral-equivalence-D} to
\(\widetilde{\mathcal A}\). We then use
\eqref{eq:exact-atlas-norm-identity-D} and apply the same norm equivalence
to \(\mathcal A\). This gives
\[
 \begin{aligned}
 \|f\circ T\|_{H^k(\Sph^2)}
 &\le
 c_{\widetilde{\mathcal A},k}^{-1}
 \|f\circ T\|_{H^k(\Sph^2),\widetilde{\mathcal A}}
 \\
 &=
 c_{\widetilde{\mathcal A},k}^{-1}
 \|f\|_{H^k(\Sph^2),\mathcal A}
 \\
 &\le
 c_{\widetilde{\mathcal A},k}^{-1}
 C_{\mathcal A,k}
 \|f\|_{H^k(\Sph^2)}.
 \end{aligned}
\]
Thus \(C_T\) is bounded on \(H^k(\Sph^2)\) for smooth functions.

Applying the same estimate to \(T^{-1}\), with \(f\circ T\) in place of
\(f\), gives the reverse bound. Hence
\[
 \|f\circ T\|_{H^k(\Sph^2)}
 \asymp_{k,T}
 \|f\|_{H^k(\Sph^2)}
\]
for \(f\in C^\infty(\Sph^2)\). By the density of
\(C^\infty(\Sph^2)\) in \(H^k(\Sph^2)\), the operator extends uniquely to
all of \(H^k(\Sph^2)\); see Taylor~\cite[Ch.~4]{Taylor2011}. The same
argument for \(T^{-1}\) shows that the extension is invertible and
\[
 C_T^{-1}=C_{T^{-1}}.
\]
\end{proof}

%% ------------------------------------------------------------
%% Subsection: Fractional Sobolev orders and duality
%% ------------------------------------------------------------
\subsection{Fractional Sobolev orders and duality}
\label{subsec:fractional-sobolev-stability-appendix}

We next extend the integer-order result to arbitrary positive orders by
interpolation.

\begin{corollary}
\label{cor:fractional-sobolev-stability-appendix}
For every \(\tau>0\), composition with \(T\) is a bounded isomorphism on
\(H^\tau(\Sph^2)\).
\end{corollary}

\begin{proof}
Choose an integer \(r>\tau\) and set \(\theta:=\tau/r\). The interpolation
identity for Sobolev spaces on compact manifolds gives
\[
[H^0(\Sph^2),H^r(\Sph^2)]_\theta
=
H^{\theta r}(\Sph^2)
=
H^\tau(\Sph^2);
\]
see Taylor~\cite[Ch.~4]{Taylor2011}.

By Theorem~\ref{thm:sobolev-stability-appendix}, \(C_T\) is bounded on
both endpoint spaces. Interpolation therefore gives boundedness on
\(H^\tau(\Sph^2)\). Applying the same argument to \(T^{-1}\) gives the
bounded inverse \(C_{T^{-1}}\).
\end{proof}

We use the following duality consequence in
Corollary~\ref{cor:sobolev-jones-bound}.

\begin{remark}
\label{rem:dual-norm-comparison}
Sobolev duality gives
\[
\bigl(H^\tau(\Sph^2)\bigr)^*
\equiv
H^{-\tau}(\Sph^2);
\]
see \cite[Ch.~4]{Taylor2011}. Consequently, if \(X\) is
\(H^\tau(\Sph^2)\) equipped with an equivalent norm, written as
\[
X\equiv H^\tau(\Sph^2),
\]
then
\[
X^*\equiv H^{-\tau}(\Sph^2).
\]
Equivalently, there are constants \(c_*,C_*>0\) such that
\[
c_*\|F\|_{H^{-\tau}(\Sph^2)}
\le
\|F\|_{X^*}
\le
C_*\|F\|_{H^{-\tau}(\Sph^2)}
\]
for every \(F\in H^{-\tau}(\Sph^2)\).
\end{remark}

%% ------------------------------------------------------------
%% End of document
%% ------------------------------------------------------------
\end{document}